\documentclass[11pt]{amsart}
\usepackage{amssymb,amsmath,amsthm,amsfonts}
\usepackage{latexsym}
\usepackage{bbm,enumerate}
\usepackage{mathtools}
\usepackage{hyperref}
\hypersetup{%
    colorlinks = false,
    hidelinks
}

\usepackage{marginnote}
\usepackage[letterpaper, top=1.2in, bottom=1.2in, outer=1.1in, inner=1.1in, heightrounded, marginparwidth=2.5cm, marginparsep=2mm]{geometry}

\allowdisplaybreaks
\newcommand{\nc}{\newcommand}
\newcommand{\rnc}{\renewcommand}
\nc{\les}{\lesssim}
\nc{\nit}{\noindent}
\nc{\nn}{\nonumber}
\nc{\D}{\partial}
\nc{\p}{\partial}
\nc{\diff}[2]{\frac{d #1}{d #2}}
\nc{\diffn}[3]{\frac{d^{#3} #1}{d {#2}^{#3}}}
\nc{\pdiff}[2]{\frac{\partial #1}{\partial #2}}
\nc{\pdiffn}[3]{\frac{\partial^{#3} #1}{\partial{#2}^{#3}}}
\nc{\abs}[1]{\left|#1\right|}
\nc{\cAc}{{\cal A}_c}
\nc{\cE}{{\cal E}}
\nc{\cF}{{\cal F}}
\nc{\cP}{{\cal P}}
\nc{\cV}{{\cal V}}
\nc{\cQ}{{\cal Q}}
\nc{\cGin}{{\cal G}_{\rm in}}
\nc{\cGout}{{\cal G}_{\rm out}}
\nc{\cO}{{\cal O}}
\nc{\Lav}{{\cal L}_{\rm av}}
\nc{\cL}{{\cal L}}
\nc{\cB}{{\cal B}}
\nc{\cZ}{{\cal Z}}
\nc{\cR}{{\cal R}}
\nc{\cT}{{\cal T}}
\nc{\cY}{{\cal Y}}
\nc{\cX}{{\cal X}}
\nc{\cXT}{{{\cal X}(T)}}
\nc{\cBT}{{{\cal B}(T)}}
\nc{\vD}{{\vec \mathcal{D}}}
\nc{\efield}{\mathcal{E}}
\nc{\vE}{{\vec \efield}}
\nc{\vB}{{\vec \mathcal{B}}}
\nc{\vH}{{\vec \mathcal{H}}}
\nc{\mR}{\mathcal{R}}
\nc{\mG}{\mathcal{G}}
\nc{\mE}{\mathcal{E}}
\nc{\ty}{{\tilde y}}
\nc{\tu}{{\tilde u}}
\nc{\tV}{{\tilde V}}
\nc{\Pc}{{\bf P_c}}
\nc{\bx}{{\bf x}}
\nc{\bX}{{\bf X}}
\nc{\bXYZ}{{\bf XYZ}}
\nc{\bY}{{\bf Y}}
\nc{\bF}{{\bf F}}
\nc{\bS}{{\bf S}}
\nc{\dV}{{\delta V}}
\nc{\dE}{{\delta E}}
\nc{\TT}{{\Theta}}
\nc{\dPsi}{{\delta\Psi}}
\nc{\order}{{\cal O}}
\nc{\Rout}{R_{\rm out}}
\nc{\eplus}{e_+}
\nc{\eminus}{e_-}
\nc{\epm}{e_\pm}
\nc{\eps}{\varepsilon}
\nc{\vnabla}{{\vec\nabla}}
\nc{\G}{\Gamma}
\nc{\w}{\omega}
\nc{\mh}{h}
\nc{\mg}{g}
\nc{\vphi}{\varphi}
\nc{\tlambda}{\tilde\lambda}
\nc{\be}{\begin{equation}}
\nc{\ee}{\end{equation}}
\nc{\ba}{\begin{eqnarray}}
\nc{\ea}{\end{eqnarray}}
\rnc{\k}{\kappa}
\rnc{\d}{\delta}
\rnc{\l}{\lambda}
\rnc{\rq}{{\rho_q}}
\nc{\dist}{\text{dist}}
\nc{\g}{\gamma}
\nc{\ol}{\overline}
\nc{\n}{\nu}
\newtheorem{theorem}{Theorem}[section]
\newtheorem{lemma}[theorem]{Lemma}
\newtheorem{prop}[theorem]{Proposition}
\newtheorem{corollary}[theorem]{Corollary}
\newtheorem{defin}[theorem]{Definition}
\newtheorem{rmk}[theorem]{Remark}

\nc{\pT}{\partial_T}
\nc{\pz}{\partial_z}
\nc{\pt}{\partial_t}
\nc{\la}{\langle}
\nc{\ra}{\rangle}
\nc{\infint}{\int_{-\infty}^{\infty}}
\nc{\halfwidth}{6.5cm}
\nc{\figwidth}{10cm}
\nc{\f}{\frac}
\nc{\comment}[1]{{\color{red}\sl$\langle$#1$\rangle$}}
\nc{\nlayers}{L} \nc{\nsectors}{M}
\nc{\indicator}{\mathbf{1}}
\nc{\Rhole}{R_{\rm hole}}
\nc{\Rring}{R_{\rm ring}}
\nc{\neff}{n_{\rm eff}}
\nc{\Frem}{F_{\rm rem}}
\nc{\R}{\mathbb R}
\nc{\C}{\mathbb C}
\nc{\Z}{\mathbb Z}
\nc{\N}{\mathbb N}
\nc{\DD}{\Delta}
\nc{\cD}{\mathcal D}
\nc{\lnorm}{\left\|}
\nc{\rnorm}{\right\|}

\usepackage{microtype}
\nc{\rnormp}{\right\|_{\ell^{p,\eps}}}
\nc{\rar}{\rightarrow}
\nc{\1}{\mathbbm{1}}
\nc{\mf}[1]{\mathfrak{#1}}
\nc{\mc}[1]{\mathcal{#1}}
\rnc{\t}[1]{\text{#1}}
\nc{\ang}[1]{\left\langle#1\right\rangle}
\nc{\set}[1]{\left\{#1\right\}}
\DeclareMathOperator{\Ker}{ker}

\begin{document}

\begin{abstract}
	We investigate dispersive estimates for the massless three dimensional Dirac  equation with a potential.  In particular, we show that the Dirac evolution satisfies a $\la t\ra^{-1}$ decay rate as an operator from $L^1$ to $L^\infty$ regardless of the existence of zero energy eigenfunctions.  We also show this decay rate may be improved to $\la t\ra ^{-1-\gamma}$ for any $0\leq \gamma < 1/2$ at the cost of spatial weights.  This estimate, along with the $L^2$ conservation law allows one to deduce a family of Strichartz estimates in the case of a threshold eigenvalue.  We classify the structure of threshold obstructions as being composed of zero energy eigenfunctions. Finally, we show the Dirac evolution is bounded for all time with minimal requirements on the decay of the potential and smoothness of initial data. 
\end{abstract}

\title[The Massless Dirac Equation in Three Dimensions]{\textit{The Massless Dirac Equation in Three Dimensions: Dispersive estimates and zero energy obstructions}}

\author[W.~R. Green, C.~Lane, B.~Lyons, S.~Ravishankar, A.~Shaw]{William~R. Green, Connor Lane, Benjamin Lyons, Shyam Ravishankar, Aden Shaw}
\thanks{The first author was partially supported by Simons Foundation Grant 511825.}
\address{Department of Mathematics\\
Rose-Hulman Institute of Technology \\
Terre Haute, IN 47803, U.S.A.}
\email{green@rose-hulman.edu, lanecf@rose-hulman.edu, lyonsba1@rose-hulman.edu, ravishs@rose-hulman.edu, shawap@rose-hulman.edu}

\address{231 West 18th Avenue\\
	Columbus, OH 43210-1174 }
\email{shaw.1287@osu.edu}
\subjclass{35Q41, 35L40, 47B15}

\maketitle
\section{Introduction}
We consider the linear Dirac equation with a potential:
\begin{align*}%\label{eqn:Dirac}
	i\partial_t \psi(x, t) = (D_m + V(x))\psi(x, t), \qquad
	\psi(x, 0) = \psi_0(x).
\end{align*}
Here, $x \in \mathbb R^3$ is the spatial variable and $\psi(x,t) \in \mathbb C^{4}$. The free Dirac operator
$D_m$ is defined by
\begin{align*}%\label{eqn:Dmdef}
	D_m = -i \alpha \cdot \nabla + m \beta = -i \sum_{k = 1}^{3} \alpha_k \partial_{k} + m \beta,
\end{align*}
where $m\geq0$ is a constant, and the $4\times 4$ Hermitian matrices $\alpha_0 \coloneqq \beta$ and $\alpha_j$ satisfy
\begin{align}\label{eqn:anticomm}
	\alpha_j \alpha_k + \alpha_k\alpha_j 
	= 2\delta_{jk} \1_{\mathbb C^{4}}, 
	\qquad 
	\t{for all } j, k \in \{0, 1, 2, 3\}.
\end{align}
We consider the massless case, when $m = 0$, which may be used to model the dynamics of massless Fermions.  The Dirac equation more generally is a model for relativistic dynamics of quantum particles.  We refer to the short introductory article, \cite{EGTwhat}, or the monograph of Thaller, \cite{Thaller}, for more detailed introductions to the Dirac equation. For concreteness, in three dimensions we use
\begin{align*}
	\beta = \left[\begin{array}{cc} I_{\mathbb{C}^2} & 0\\ 0 & -I_{\mathbb{C}^2}
	\end{array}\right], \ \alpha_i=\left[\begin{array}{cc} 0 & \sigma_i \\ \sigma_i & 0
	\end{array}\right],
\end{align*}
\begin{align*}
	\sigma_1=\left[\begin{array}{cc} 0 & -i\\ i & 0
	\end{array}\right],\
	\sigma_2=\left[\begin{array}{cc} 0 & 1 \\ 1 & 0
	\end{array}\right],\
	\sigma_3=\left[\begin{array}{cc} 1 & 0 \\ 0 & -1
	\end{array}\right].
\end{align*}

The following identity,\footnote{When discussing  scalar operators such as $-\Delta+m^2-\lambda^2$ in the context of the Dirac equation they are to be understood as matrix-valued, that is as $(-\Delta+m^2-\lambda^2)\mathbbm 1_{\mathbb C^{4}}$.} which follows from   \eqref{eqn:anticomm},
\be  \label{dirac_schro_free}
(D_m-\lambda \mathbbm 1)(D_m+\lambda \mathbbm 1) =(-i \alpha \cdot \nabla +m\beta -\lambda \mathbbm 1)
(-i\alpha\cdot \nabla+m\beta+\lambda \mathbbm 1)   =(-\Delta+m^2-\lambda^2) 
\ee
allows us to formally define the free Dirac resolvent
operator $\mathcal R_0(\lambda)=(D_m-\lambda)^{-1}$ in terms of the
free resolvent $R_0(\lambda)=(-\Delta-\lambda)^{-1}$ of  the Schr\"{o}dinger operator for $\lambda$ in the resolvent set:
\begin{align}\label{eqn:resolvdef}
	\mathcal R_0(\lambda)=(D_m+\lambda) R_0(\lambda^2-m^2).
\end{align}
For our purposes, when $m=0$, we have
\begin{align*}
	\mathcal R_0(\lambda)=(-i\alpha \cdot \nabla+\lambda) R_0(\lambda^2)= (D_0+\lambda)R_0(\lambda^2).
\end{align*}

To state our main theorem, we introduce the following notation.  We let $\la x\ra=(1+|x|^2)^{1/2}$, let $a-$ denote $a-\epsilon$ for an arbitrarily small, but fixed $\epsilon>0$.  Similarly, $a+=a+\epsilon$.  We write $A\les B$ if there is an absolute constant $C>0$ so that $A\leq CB$.  For matrix-valued functions if $|V_{ij}(x)|\les \la x\ra^{-\delta}$ for every entry, we write $|V(x)|\les \la x\ra^{-\delta}$.  We define $\chi(\lambda)$ to be a smooth even cut-off to a sufficiently small neighborhood of zero.  Similarly, any function space $X$ used in the paper denotes the space of $\mathbb C^4$-valued functions with all entries in $X$.  That is, by $f\in L^1(\R^3)$ we mean $f(x)=(f_j(x))_{j=1}^4$ with each component an $L^1$ function. We define the polynomially weighted spaces $L^{p,\sigma}=\{f : \la\, \cdot\, \ra^\sigma f\in L^p\}$. We call the threshold zero energy regular if there are no zero-energy eigenfunctions of the Dirac operator $\mathcal H \coloneqq D_0 + V$.
Our main results control the evolution of the Dirac operator.
\begin{theorem}\label{thm:main}
	Assume that $V$ is self-adjoint and $|V(x)|\les \la x\ra^{-\delta}$.
	\begin{enumerate}[i)]
		\item Assume that zero is regular, for fixed $0\leq \gamma \leq 1$, if $\delta>3+2\gamma$ we have
		$$
		\|e^{-it\mathcal H}\chi(\mathcal H)\|_{L^{1,\gamma}\to L^{\infty,-\gamma}} \les \la t\ra^{-1-\gamma}.
		$$
		\item If zero is not regular, then for fixed $0\leq \gamma < 1/2$, if $\delta>3+4\gamma$ we have
		$$
		\|e^{-it\mathcal H}\chi(\mathcal H)\|_{L^{1,\gamma}\to L^{\infty,-\gamma}} \les \la t\ra^{-1-\gamma}.
		$$
		
	\end{enumerate}

\end{theorem}

One main novelty of these results is that the same time decay bounds hold regardless of the regularity of the threshold.  This phenomenon, to the best of the authors' knowledge, is found only in one-dimensional cases  without additional assumptions on the structure of the threshold eigenspace, see \cite{GS,egd1} for example.  A similar bound holds in the case of the two-dimensional massless case, \cite{EGG2}, though one must first subtract off a finite rank operator that decays more slowly for large $t$.   

For completeness, we pair these bounds with an argument that controls the high energy portion of the evolution.  We define $P_{ac}$ to be the projection onto the absolutely continuous subspace of $L^2$.  We obtain the following family of bounds.
\begin{theorem}\label{thm:full results}
	
	Assume that $V$ is self-adjoint with continuous entries satisfying $|V(x)|\les \la x\ra^{-\delta}$.   For any fixed $0\leq \gamma \leq 1$, we have
	$$
	\|e^{-it\mathcal H} P_{ac}(\mathcal H) \la \mathcal H\ra^{-3-}\|_{L^{1,\gamma}\to L^{\infty,-\gamma}} \les \la t\ra^{-1-\gamma},
	$$
	provided $\delta>3+2\gamma$ if zero is regular. If there is an eigenvalue at zero, then for  any $0\leq\gamma< 1/2$, we require $\delta > 3+4\gamma$ .
	
\end{theorem}
We establish the weighted bounds by developing appropriate representations of the spectral measure associated with the perturbed evolution.  
Finally, as an application of the Lipschitz continuity properties of the spectral measure that we develop, we obtain weaker bounds with minimal decay assumptions on the potential.

\begin{theorem}\label{thm:weak results}
	Assume that $V$ is self-adjoint with continuous entries satisfying $|V(x)|\les \la x\ra^{-\delta}$ for some $\delta>1$.  If zero energy is regular, then
	$$
	\|e^{-it\mathcal H} P_{ac}(\mathcal H) \la \mathcal H\ra^{-3-}\|_{L^{1}\to L^{\infty}} \les 1.
	$$
\end{theorem}
We note that this result is essentially sharp with respect to differentiability of the initial data and required decay at infinity.  The free Dirac evolution requires the same amount of differentiability, see Theorem~\ref{thm:free} below.

The class of potentials considered here are of the form obtained by linearizing about a standing wave solution to a nonlinear Dirac equation, see \cite{BCbook} for example.
Our estimates follow by treating the Dirac evolution as an element of the functional calculus.  For the potentials we consider $\mathcal H=D_0+V$ is self-adjoint, and $\sigma(\mathcal H)=\R$.  Further, there is a Weyl criterion, $\sigma_{c}(\mathcal H)=\sigma_c(D_0)=\R$, and there is no singularly continuous spectrum or embedded eigenvalues, \cite{GM,BC}.  The spectral measure may be constructed in terms of the limiting resolvent operators
$$
\mR_V^\pm(\lambda)=\lim_{\epsilon\to 0^+} [D_0+V-(\lambda\pm i\epsilon)]^{-1}.
$$
These operators are well-defined by Agmon's limiting absorption principle as operators on weighted $L^2$ spaces, \cite{agmon}, and their relationship to the Schr\"{o}dinger resolvents \eqref{eqn:resolvdef}.  The difference of these limiting operators provide the spectral measure.  Specifically, the Stone's formula allows us to express the evolution of the solution operator as
\begin{align}\label{eqn:stone}
	e^{-it\mathcal H}P_{ac}(\mathcal H)f=\frac{1}{2\pi i} \int_{\R} e^{-it\lambda} [\mR_V^+-\mR_V^-](\lambda) f\, d\lambda.
\end{align}
Our methods seek to understand the perturbed resolvents $\mR_V^\pm$ as perturbations of the free resolvent operators
$$
\mR_0^\pm(\lambda) = \lim_{\epsilon\to 0^+} [D_0-(\lambda\pm i\epsilon)]^{-1}.
$$
Using \eqref{dirac_schro_free}, we obtain the identity $\mR_0^\pm (\lambda) =(D_0+\lambda)R_0^\pm(\lambda)$ where $R_0^\pm(\lambda)=(-\Delta-(\lambda\pm i0))^{-1}$ are the limiting resolvent operators for the Schr\"{o}dinger operator.  Due to the well-known explicit formulas for these, one may write explicit formulas for the Dirac resolvent, see \eqref{eqn:R0 ugly} below.   

Global dispersive bounds that seek to control the $L^\infty$ size of solutions are well-studied in the Schr\"odinger, wave, and Klein-Gordon contexts, see the recent survey paper \cite{Scsurvey}.  The dispersive estimates for the three dimensional Dirac equation is more studied going back to the work of Boussa\"id \cite{Bouss1}, and D'Ancona and Fanelli, \cite{DF} in the massive $m>0$ in the weighted-$L^2$ setting.  The characterization of threshold obstructions as resonances and eigenvalues along with their effect on the dispersive bounds in two and three dimensions has  been studied by  Erdo\smash{\u{g}}an and the first author in \cite{egd}, along with Toprak, \cite{EGT,EGT2} in the massive case, and with Goldberg in the massless case \cite{EGG2}.  More recent work studied the dispersive bounds in the one dimensional case by Erdo\smash{\u{g}}an and the first author in \cite{egd1}. See also the recent work of Kraisler, Sagiv and Weinstein, \cite{KSW}, which considered non-compact perturbations in one dimension.   Much of the work  relies on the techniques developed in the study of other dispersive equations, notably the Schr\"odinger equation \cite{Mur,GS,ES,KS,eg2}, which analyze the effect of threshold energy obstructions. 

Nonlinear Dirac equations have also garnered interest.  See for example, \cite{EV,PS,BH3,BH,CTS,BC2}, and Boussa\"id and Comech's monograph, \cite{BCbook}.  There is a longer history in the study of spectral properties of Dirac operators.  Limiting absorption principles for the Dirac operators have been studied in \cite{Yam,GM,BG,EGG,CGetal}.    The lack of embedded eigenvalues, singular continuous spectrum and other spectral properties is well established, \cite{BG1,GM,MY,CGetal,BC}. 

There has been recent work on the massless three dimensional Dirac equation with a Coulomb potential of the from $\nu/|x|$.  Here one must restrict the value of $\nu$ to an appropriate interval to ensure there is a self-adjoint extension of the Dirac operator.  Danesi, \cite{Dan}, and separately Cacciafesta, S\'er\'e and Zhang, \cite{CSZ} establish various families of Strichartz estimates for these operators with Coulomb potentials.

Strichartz estimates for potentials of this form are known when zero is regular, \cite{EGG}, in both the massive and massless cases.  By interpolating the bound in Theorem~\ref{thm:full results} with the $L^2$ conservation law, one obtains a family of $L^p$ dispersive bounds of the form
\begin{align}
	\|e^{-it\mathcal H} \la \mathcal H\ra^{\frac32-\frac{3}{p}-}P_{ac}(\mathcal H)f\|_{L^{p'}}\les |t|^{\frac{3}{2}-\frac{3}{p}} \|f\|_{L^p}, \qquad 1\leq p\leq 2.
\end{align}
As in the classic paper of Ginibre and Velo, \cite{GV}, these may be used to deduce Strichartz estimates.
\begin{corollary}\label{cor:strich}
	Assume that $V$ is self-adjoint with continuous entries satisfying $|V(x)|\les \la x\ra^{-\delta}$ for some $\delta>3$. If zero energy is not regular, one has
	$$
	\| e^{-it\mathcal H} \la \mathcal H\ra^{\frac3{r}-\frac{3}{2}-}P_{ac}(\mathcal H)f\|_{L^q_t L^r_x}\les \|f\|_{L^2}, \qquad \frac{2}{q}+\frac{2}{r}= 1, \quad q>2, \, 2\leq r\leq \infty.
	$$
\end{corollary}

The paper is organized as follows.  First in Section~\ref{sec:free} we establish the natural decay properties for the free massless Dirac evolution.  In Section~\ref{sec:low expansions} we develop expansions of the limiting free resolvent operators in a neighborhood of the threshold, and use them to establish continuity and differentiability properties of the spectral measure near the threshold.  In Section~\ref{sec:zero reg} we prove the low energy dispersive bounds in Theorem~\ref{thm:main} when zero is regular.  As a further application of the Lipschitz continuity of the spectral measure, we prove families of time-integrable estimates on polynomially weighted space in Subsection~\ref{subsec:wtd bds}.  In Section~\ref{sec:eval} we show that the same estimates hold even if there is a threshold eigenvalue at the cost of further decay of the potential.  In Section~\ref{sec:thresholds} we characterize the existence of zero energy eigenvalues and relate them to the spectral measure constructed in Section~\ref{sec:low expansions}.  Finally in Section~\ref{sec:hi} we control the high energy portion of the evolution.

\section{Free Dirac dispersive estimates}\label{sec:free}

To understand the dynamics of the perturbed solution operator $e^{-it\mathcal H}P_{ac}(\mathcal H)$, we first study the dynamics of the free solution operator $e^{-itD_0}$ when $V\equiv 0$.  In this section we develop the needed oscillatory integral estimates to understand the free evolution, as well as prove several estimates about the dynamics of solutions to the free massless Dirac equation.

Due to the relationship between the massless free Dirac equation and the
free wave equation, $D_0^2 f=-\Delta f$,
we can expect a natural time decay rate of size $|t|^{-1}$ 
as one has in the wave equation  provided the initial data has more than $2$ weak derivatives in $L^1(\R^2)$.
In the case of Dirac equation, as in Schr\"{o}dinger equation, the time decay for large $|t|$ may be improved  at the cost of spatial weights. 

Much of our low energy analysis will rely on relationships between smoothness of a function and the decay of its Fourier Transform, which we encapsulate in the following.

\begin{lemma}\label{lem:osc int}
	
	Let $\mathcal E(\lambda)$ be a a function supported on $(-1,1)$ with $\mathcal E(\lambda)$ bounded and $\mathcal E'(\lambda)\in L^1$.
	Then, we have the bound
	\begin{align*}
		\bigg|\int_{\R} e^{-it\lambda} \mathcal E(\lambda) \, d\lambda \bigg|
		\les \frac{1}{|t|} \int_{\R} \bigg|\mathcal E'(\lambda) - \mE'\!\left(\lambda-\frac{\pi}{t}\right)\!\bigg| \, d\lambda.
	\end{align*}
\end{lemma}

\begin{proof}
	Since $\mathcal E(\lambda)$ is bounded and $\mathcal E'\in L^1$ we may integrate by parts once, the support of $\mathcal E$ ensures there are no boundary terms,
	\begin{align*}
		\int_{\R} e^{-it\lambda} \mathcal E(\lambda) \,d\lambda
		= \frac{1}{it} \int_{\R}e^{-it\lambda} \mathcal E'(\lambda) \, d\lambda.
	\end{align*}
	We note that if $|t|<1$ then $\mathcal E'(\lambda - \pi/t) = 0$ on the support of $\mathcal E(\lambda)$. Applying the triangle inequality to the equality above proves the claim.
	For $|t| > 1$ large, we use the support of $\mathcal E$, a change of variables $\lambda \mapsto \lambda - \pi/t$, and then the fact that $e^{i\pi} = -1$ to write
	\begin{align*}
		\int_{\R} e^{-it\lambda} \mE'(\lambda) \,d\lambda
		= \int_{\R}e^{-it(\lambda-\pi/t)}\mE'(\lambda-\pi/t) \, d\lambda
		= -\int_{\R} e^{-it\lambda}\mE'\!\left(\lambda-\frac{\pi}{t}\right)  d\lambda.
	\end{align*}
	Then using the triangle inequality, we have
	\begin{align*}
		\left|\int_{\R} e^{-it\lambda} \mathcal E(\lambda) \, d\lambda\right|
		& = \left|\frac{1}{it} \int_{\R} e^{-it\lambda} \mathcal E'(\lambda) \, d\lambda \right|  = \left|\frac{1}{2it} \int_{\R} e^{-it\lambda} \bigg[ \mathcal E'(\lambda)-\mE'\!\left(\lambda-\frac{\pi}{t}\right)\!\bigg]\, d\lambda \right| \\
		& \les \frac{1}{|t|} \int_{\R} \bigg|\mathcal E'(\lambda) - \mE'\!\left(\lambda-\frac{\pi}{t}\right)\!\bigg| \, d\lambda
	\end{align*}
	as desired. 
\end{proof}

Using this oscillatory integral bound, we can prove dispersive bounds for the free Dirac evolution.

\begin{theorem}\label{thm:free} 
	We have the estimate
	\begin{align*}
		{\left\| e^{-it D_0} \la D_0 \ra^{-2-} 
			\right\|}_{L^1\to L^\infty} \les |t|^{-1}.
	\end{align*}
	Furthermore,  for $|t| > 1$ and  any $0\leq \gamma \leq 1$, we have
	\begin{align*}
		\left\|\la x\ra^{-\gamma} e^{-it D_0} \la D_0 \ra^{-2-} \la y\ra^{-\gamma}
		\right\|_{L^1\to L^\infty} \les | t|^{-1-\gamma}.
	\end{align*}
	
\end{theorem}

\begin{proof}
	First note that the Stone's formula, \eqref{eqn:stone}, for the free evolution is
	\begin{align}\label{eq:Stone}
		e^{-itD_0} =\int_{\R}  e^{-it\lambda} [\mathcal R_0^+ -\mR_0^-](\lambda)(x,y) \, d\lambda.
	\end{align}
	Write $\mu(\lambda, x,y) \coloneqq[\mR_0^+-\mR_0^-](\lambda)( x, y)$. Utilizing \eqref{eqn:resolvdef} and the well-known expansion for the limiting resolvents of the free Schr\"{o}dinger in three dimensions, see \eqref{eqn:R0 ugly} below, we have
	\begin{align*}
		\mu(\lambda, x,y) = (-i\alpha \cdot \nabla +\lambda) \bigg( \frac{i\sin(\lambda |x-y|)}{2\pi |x-y|} \bigg).
	\end{align*}
	Similar to the estimates we establish in Lemma~\ref{lem:R0 expansions} below, we have the following bounds: 
	\begin{align}\label{eqn:mu0 bounds} 
		|\mu(\lambda,x,y)| \les \min\!\left(|\lambda|^2, \frac{|\lambda|}{|x - y|}\right),
		\quad
		|\partial_\lambda \mu(\lambda,x,y)| \les |\lambda|, 
		\quad
		|\partial_\lambda^2 \mu(\lambda,x,y)| \les 1 + |\lambda|\,|x-y|.
	\end{align}
	For the first bound  (with $r = |x - y|$ and $\hat{e} = \nabla_x |x-y| = (x-y)/|x-y|$ the unit vector in the $x-y$ direction) we have
	\begin{align*}
		\mu(\lambda,x,y)=\alpha \cdot \hat{e}\, \lambda^2 \bigg( \frac{\lambda r\cos(\lambda r)-\sin(\lambda r)}{2\pi (\lambda r)^2} \bigg) +\frac{i\lambda \sin(\lambda r)}{2\pi r}.
	\end{align*}
	This is bounded by $|\lambda|/r$, which we obtain by taking $|\!\cos(\lambda r)| \les 1$ and $|\!\sin(\lambda r)| \les |\lambda r|$ in the first term and $|\!\sin(\lambda r)| \les 1$ in the second term. When $|\lambda |r \geq 1$, this is bounded by $|\lambda|^2$ as we may freely multiply by $|\lambda |r$ on the upper bound.  When $|\lambda| r<1$, we utilize the cancellation of the numerator of the first term up to order $(\lambda r)^3$ to see this is bounded by $|\lambda|^3r\leq |\lambda|^2$ as desired.  The contribution of the second term is more easily seen to be of size $|\lambda|^2$ using $|\!\sin(\lambda r)|\les |\lambda|r$.  These bounds hold for any $\lambda$.

	For the derivative, we note that
	\begin{align*}
		2\pi \partial_\lambda\mu (\lambda, x,y)
        & = -i\alpha \cdot \nabla  \big( i\cos(\lambda r) \big)+i\lambda\cos(\lambda r)+\frac{i \sin (\lambda r)}{r}\\
		& = (\alpha\cdot \hat{e}) \lambda\sin(\lambda r)+i\lambda\cos(\lambda r)+\frac{i \sin (\lambda r)}{r}
	\end{align*}
	This is easily seen to be bounded by $|\lambda|$.  Finally, 
	\begin{align*}
		|2\pi \partial_\lambda^2\mu (\lambda, x,y)| & = |(\alpha\cdot \hat{e})(\sin(\lambda r) + \lambda r\cos(\lambda r)) +2i\cos(\lambda r)-i\lambda r\sin(\lambda r)| \les 1+|\lambda|\, r.
	\end{align*}
	For the low energy, we define $\mu_0(\lambda,x,y)=\chi(\lambda)\mu(\lambda,x,y)$, so
	$\mu_0$ and its derivatives satisfy the corresponding bounds above for derivatives of $\mu$ in \eqref{eqn:mu0 bounds}. If any derivative acts on the cut-off $\chi(\lambda)$, we note that $\chi'(\lambda)$ is supported on the annulus $|\lambda|\approx 1$, so that $|\chi^{(k)}(\lambda)| \les 1$ or $|\lambda|^{-k}$ as needed.

	Using \eqref{eqn:mu0 bounds} and the support of $\chi$, which is contained in $[-1,1]$, we have
	\begin{align*}
		\bigg| \int_{\R} e^{-it\lambda} \mu_0(\lambda,x,y) \, d\lambda\bigg|  \les \int_{-1}^1 \lambda^2\, \, d\lambda,
	\end{align*}
	so this integral is bounded uniformly in $x, y$.  To obtain time decay, we integrate by parts once. Again we use \eqref{eqn:mu0 bounds}, this time to ensure there are no boundary terms  
	\begin{align*}
		\bigg| \int_{\R} e^{-it\lambda} \mu_0(\lambda,x,y) \, d\lambda\bigg|
		= \bigg|\frac{1}{it} \int_{\R} e^{-it\lambda} \partial_\lambda \mu_0(\lambda,x,y)  \, d\lambda\bigg|
		\les \frac{1}{|t|} \int_{-1}^1  |\lambda|\, d\lambda 
		\les \frac{1}{|t|}.
	\end{align*}
	Here, \eqref{eqn:mu0 bounds} was used along with the support of $\chi$.  For the low energy contribution, we note that the bounds of $1$ and $|t|^{-1}$ show that it is bounded by $\la t\ra^{-1}$. That is, the low energy contribution is bounded for all times.
	
	For the weighted bound, rather than integrate by parts again, we employ an argument based on Lipschitz continuity, which will be helpful for the perturbed case.
	Take $|\lambda_1|\leq |\lambda_2|\leq 1$.  We claim the following bound holds on the support of $\chi$:
	\begin{align}\label{eqf:mu Lipschitz}
		|\mu_0(\lambda_2,x,y) - \mu_0(\lambda_1,x,y)| 
		\les |\lambda_2|^{2-\gamma} |\lambda_2-\lambda_1|^{\gamma}, 
		\qquad 0\leq \gamma\leq 1.
	\end{align}
	By the triangle inequality and \eqref{eqn:mu0 bounds} we have
	\begin{align*}
		|\mu_0(\lambda_2,x,y) - \mu_0(\lambda_1,x,y)| 
		\les |\lambda_2|^2.
	\end{align*}
	On the other hand if we apply the mean value theorem and \eqref{eqn:mu0 bounds}, writing $I=[\min(\lambda_1,\lambda_2),\max(\lambda_1,\lambda_2)]$, we see
	\begin{align*}
		|\mu_0(\lambda_2,x,y)-\mu_0(\lambda_1,x,y)| 
		& = \bigg| \int_{\lambda_1}^{\lambda_2} \partial_\lambda \mu_0(s, x,y)\, ds \bigg| \\
		& \les |\lambda_2-\lambda_1| \sup_{s \in I} |\partial_\lambda \mu_0(s,x,y)| 
		\les |\lambda_2|\, |\lambda_2-\lambda_1|.
	\end{align*}
	Interpolating between these two bounds gives the claim in \eqref{eqf:mu Lipschitz}.
	
	On the support of $\chi$, a similar argument shows that
	\begin{align}\label{eqf:mu Lipschitz2}
		|\partial_\lambda \mu_0(\lambda_2,x,y)-\partial_\lambda \mu_0(\lambda_1,x,y)| 
		\les |\lambda_2|^{1-\gamma} |\lambda_2-\lambda_1|^{\gamma} (1+|\lambda_2|\,|x-y|)^\gamma, 
		\qquad 0 \leq \gamma 
		\leq 1.
	\end{align}
	
	Now, for the weighted bound if $|t|>1$ we apply Lemma~\ref{lem:osc int}
	(using \eqref{eqf:mu Lipschitz2} with $\mathcal E=\mu_0$, $\lambda_j$ one of $\lambda$ and $\lambda - \pi/t$) to see
	\begin{align*}
		\left|\int_{\R} e^{-it\lambda} \mu_0(\lambda,x,y) \, d\lambda\right|
		& \les \frac{1}{|t|} \int_{-1}^1 |t|^{-\gamma} |x-y|^{\gamma}\, d\lambda \les |t|^{-1-\gamma} |x-y|^{\gamma}.
	\end{align*}
	Note that $|x-y|\les \la x\ra \la y\ra$ suffices to establish the claim for low energy, when $\lambda$ is in a neighborhood of zero.
	For large $|\lambda|$, we define the complementary cut-off $\widetilde{\chi}(\lambda) = 1 - \chi(\lambda)$. We need to bound
	\begin{align*}
		\int_{\R} e^{-it\lambda} \widetilde \chi(\lambda) \la \lambda\ra^{-2-} [\mR_0^+-\mR_0^-](\lambda, x,y) \,d\lambda.
	\end{align*}
	By \eqref{eqn:mu0 bounds} this integral is not absolutely convergent uniformly in $x,y$, but these bounds suffice to ensure there are no boundary terms when integrating by parts.
	\begin{align*}
		\bigg|\int_{\R} e^{-it\lambda} \widetilde \chi(\lambda) \la \lambda\ra^{-2-} [\mR_0^+-\mR_0^-](\lambda, x,y) \,d\lambda \bigg| 
		& \les \frac{1}{|t|}\int_{\R} \bigg| \partial_\lambda [\widetilde \chi(\lambda) \la \lambda\ra^{-2-} [\mR_0^+-\mR_0^-](\lambda, x,y)] \bigg|\, d\lambda.
    \end{align*}
    We note that, on the support of $\widetilde{\chi}$, differentiation of the first two terms is comparable to division by $\lambda$. Hence, we may bound the above integral by
    \begin{align*}
        \frac{1}{|t|} \int_{\R} \la \lambda\ra^{-1-}\, d\lambda \les \frac{1}{|t|}.
	\end{align*}
	Integrating by parts twice leads to the bound
	\begin{align*}
		\bigg|\int_{\R} e^{-it\lambda} \widetilde \chi(\lambda) \la \lambda\ra^{-2-} [\mR_0^+-\mR_0^-](\lambda, x,y) \,d\lambda \bigg| 
		& \les \frac{1}{|t|^2}\int_{\R} \bigg| \partial_\lambda^2 [\widetilde \chi(\lambda) \la \lambda\ra^{-2-} [\mR_0^+-\mR_0^-](\lambda, x,y)] \bigg|\, d\lambda\\
		& \les \frac{1}{|t|^2} \int_{\R} (\la \lambda\ra^{-2-}+\la \lambda \ra^{-1-}|x-y|)\, d\lambda 
		\les \frac{\la x-y\ra}{|t|^2}.
	\end{align*}
	Noting that $\la x-y\ra\les \la x\ra \la y\ra$ and interpolating between the two bounds establishes the claim.
\end{proof}
We utilized an argument based on Lipschitz continuity of the resolvents here rather than integrating by parts twice since such an argument is beneficial in the analysis of the perturbed operator.  Utilizing Lipschitz continuity of the perturbed resolvent in a neighborhood of $\lambda=0$ will allow us to obtain faster time decay with minimal further assumptions on the decay of the potential.

The small time blow-up is a high energy phenomenon.
We may also utilize the techniques in the proof above to obtain weaker dispersive bounds that require more smoothness on the initial data to control the high energy.
\begin{corollary}\label{cor:free weak}
	
	We have the bound
	\begin{align*}
		{\left\| e^{-it D_0} \la D_0 \ra^{-3-} 
			\right\|}_{L^1\to L^\infty} \les \la t\ra ^{-1}.
	\end{align*}
\end{corollary}

\begin{proof}
	Note that the extra power of $\la D_0\ra^{-1}$ is needed to ensure that the integral
	$$
	\bigg|\int_{\R} e^{-it\lambda} \widetilde \chi(\lambda) \la \lambda\ra^{-3-} [\mR_0^+-\mR_0^-](\lambda, x,y) \,d\lambda \bigg|, 
	$$
	converges absolutely, using $|\mu(\lambda)(x,y)|\les |\lambda|^2$.  This, combined with Theorem~\ref{thm:free} suffices to prove the claim.
\end{proof}

\section{Low Energy Resolvent Expansions and Estimates}\label{sec:low expansions}

In this section we show that, under mild decay assumptions on the potential $V$, the low energy evolution of the perturbed solution operator obeys the same bounds as the free evolution.  We do this by developing expansions for the spectral measure $[\mR_V^+-\mR_V^-](\lambda)$ in a neighborhood of the threshold to use in the Stone's formula, \eqref{eqn:stone}.  In this low energy regime, we treat $\mR_V^\pm$ as a perturbation of the free resolvent $\mR_0^\pm$, for which we obtain explicit asymptotic formulas.

To obtain expansions for the perturbed resolvent operators $\mathcal{R}^{\pm}_V(\lambda)$,  we recall the symmetric resolvent identity. As in previous analyses of the Dirac Equation, \cite{egd,EGT,EGT2,EGG2}, using that $ V \colon \R^3 \to \C^{4\times 4}$ is self-adjoint, the spectral theorem for self-adjoint matrices allows us to write 
\begin{gather*}
     V= B^{*} \Lambda B= B^{*}  |\Lambda|^{\f12} U |\Lambda|^{\f12} B =: v^{*} U v,\,\,\,\text{ where} \\
    \Lambda=\text{diag}(\lambda_1,\lambda_2,\lambda_3,\lambda_4), \text{ with }\lambda_j \in \R, \\
    |\Lambda|^{\f12}= \text{diag}(|\lambda_1|^{\f12},|\lambda_2|^{\f12},|\lambda_3|^{\f12},|\lambda_4|^{\f12}),\\
    U= \text{diag}(\text{sign}(\lambda_1), \text{sign}(\lambda_2),\text{sign}(\lambda_3),\text{sign}(\lambda_4)).
\end{gather*}
We note that if the entries of $V(x)$ are all bounded by $\la x\ra^{-\delta}$, then the entries of $v(x)$ and $v^*(x)$ are bounded by $\la x\ra^{-\delta/2}$.
Defining $ M^\pm(\lambda) = U + v\mathcal{R}^\pm_0(\lambda) v^*$,  the symmetric resolvent identity yields
\begin{align}\label{eq:pert1}
	\mathcal{R}^{\pm}_V(\lambda) =\mathcal{R}^{\pm}_0(\lambda) - \mathcal{R}^{\pm}_0(\lambda) v^{*} (M^{\pm})^{-1} (\lambda) v \mathcal{R}^{\pm}_0(\lambda). 
\end{align} 
We consider $M^\pm(\lambda)$ as a perturbation of $M^\pm(0) = U + v\mR_0^\pm(0)v^* = U + v\mG_0v^*$.  We denote this operator by $T_0 \coloneqq M^\pm(0)$.

We recall the expansion of the free resolvent for the Schr\"{o}dinger operator in $\R^3$, see \cite{ES} for example, which has
\begin{align*}
	R_0^\pm(\lambda^2)(x,y) = \frac{e^{\pm i\lambda |x-y|}}{4\pi |x-y|} = \sum_{j=0}^J (\pm i\lambda)^jG_j +o (\lambda^J),
\end{align*}
where we define the scalar-valued operators $G_j = \frac{1}{4\pi j!} |x-y|^{j-1}$.  To write expansions for the Dirac resolvent, we use \eqref{eqn:resolvdef} and note that
\begin{align}\label{eqn:R0 ugly}
	\mR_0^\pm(\lambda)(x,y)
	= (-i \alpha \cdot \nabla + \lambda)\!\left(\frac{e^{\pm i \lambda |x-y|}}{4\pi|x - y|}\right)\!
	= \frac{e^{\pm i\lambda|x-y|}}{4\pi|x-y|}\left(\alpha \cdot \hat{e} \left(\pm \lambda + \frac{i}{|x - y|}\right) + \lambda\right).
\end{align}
Recall $\hat{e}:=(x - y)/|x-y|$ is the unit vector in the $x-y$ direction. This directly gives the bounds
\begin{align}\label{eqn:R0 bounds}
	|\mR_0^\pm(\lambda,x,y)|
	\lesssim \frac{1}{|x - y|^2} + \frac{|\lambda|}{|x - y|}
	\qquad \t{and} \qquad
	|\p_\lambda \mR_0^\pm(\lambda,x,y)|
	\lesssim \frac{1}{|x - y|} + |\lambda|.
\end{align}
Now, we define operators $\mG_j$ in terms of their integral kernels:
\begin{align}\label{eqn:Gj def}
	\mG_0(x,y) = \frac{i \alpha \cdot \hat{e}}{4\pi|x-y|^2},
	\qquad 
	\mG_1(x,y) = \frac{1}{4\pi|x-y|},
	\qquad 
	\mG_2^\pm(x,y) = \pm \frac{1}{4\pi} + \frac{\alpha \cdot \hat{e}}{8\pi}.
\end{align}
Recall that these kernels are matrix-valued operators. $\mG_1$ and $\mG_2^\pm$ are $4 \times 4$  matrices, since $\alpha\cdot \hat{e} = \sum_{j = 1}^3 \alpha_j e_j$ is a $4\times 4$ matrix.  Further, we note that $\mG_1(x,y) = (-\Delta)^{-1}(x,y) \mathbbm 1_{\mathbb C^4}$.

\begin{lemma}\label{lem:R0 expansions}
	We have the following expansions  for the Dirac free resolvents:
	\begin{align*}
		\mR_0^\pm
		& = \mG_0 + \lambda \mG_1+ \mathcal E_1^\pm(\lambda,|x-y|)\\
		& = \mG_0 + \lambda \mG_1+i\lambda^2\mG_2^\pm +\mathcal E_2^\pm(\lambda,|x-y|),
	\end{align*}
	where 
	\begin{align*}
		|\mE_1^\pm(\lambda,|x-y|)| \lesssim \frac{|\lambda|}{|x-y|}, && 
		|\p_\lambda\mE_1^\pm(\lambda,|x-y|)| \lesssim  |\lambda|, \quad |\p_\lambda^2\mE_1^\pm(\lambda,|x-y|)| \lesssim 1+|\lambda|\,|x-y|.
	\end{align*}	
	Further, for any choice of $0\leq \ell \leq 1$, we have
	\begin{gather*}
	    |\mE_2^\pm(\lambda,|x-y|)| \lesssim |\lambda|^2(|\lambda|\,|x-y|)^\ell, \qquad  
		|\p_\lambda\mE_2^\pm(\lambda,|x-y|)| \lesssim  |\lambda|(|\lambda|\,|x-y|)^\ell, \\ |\p_\lambda^2\mE_2^\pm(\lambda,|x-y|)| \lesssim 1+|\lambda|\,|x-y|.\nn
	\end{gather*}
	
\end{lemma}
\begin{proof}
	% Outline: Write. $\mR_0^\pm$ and its Taylor series expansion. explicitly single out $\mG_i$ for $i = 0,1,2$ and then the bounds should be obvious.\\
	Using the Taylor expansion of $\mR_0^\pm(\lambda)(x,y)$ in $|\lambda||x-y|$, we see
	\begin{align*}
		\mR_0^\pm(\lambda)(x,y)
		& = (-i \alpha \cdot \nabla + \lambda)\left(\frac{e^{\pm i \lambda |x-y|}}{4\pi|x - y|}\right) \\
		& = \frac{i \alpha \cdot \hat{e}}{4\pi|x-y|^2} + \frac{\lambda}{4\pi|x-y|} \pm \frac{i\lambda^2}{8\pi} + \frac{i \lambda^2 \alpha \cdot \hat{e}}{8\pi} 
		+ \mathcal{E}_2^{\pm}(\lambda,|x-y|)\\
		&=\mG_0(x,y)+\lambda \mG_1(x,y)+i\lambda^2 \mG_2^\pm(x,y)+  \mathcal{E}_2^{\pm}(\lambda,|x-y|).
	\end{align*}
	where $\mathcal{E}_2^\pm$ may be differentiated freely, with differentiation comparable to division by $\lambda$.
	
	Define $\mE_1^\pm(\lambda,|x-y|)=\mR_0^\pm(\lambda)(x,y)-\mG_0(x,y)-\lambda \mG_1(x,y)$.  Denote $r:=|x-y|$, the Taylor expansion about $\lambda r=0$ gives the following bounds when $|\lambda|r < 1$:
	\begin{align}\label{eqn:otherE0}
		|\mE_1^\pm(\lambda, r)| \lesssim  |\lambda|^2,
		\qquad |\p_\lambda\mE_1^\pm(\lambda, r)| \lesssim  |\lambda| ,
		\qquad |\p_\lambda^2\mE_1^\pm(\lambda, r)| \lesssim 1.
	\end{align}
	For large $|\lambda|r$, using \eqref{eqn:R0 ugly} when $|\lambda|r\gtrsim 1$, we see that 
	$$
	\mR_0^\pm (\lambda)(x,y)-\mG_0 =\frac{e^{\pm i\lambda r}}{4\pi r} (\pm \lambda \alpha \cdot \hat{e}+\lambda) + \frac{\alpha \cdot \hat{e}}{4\pi} \bigg(\frac{e^{\pm i\lambda r}-1}{r^2} \bigg) .  
	$$
	Using that $|e^{\pm i\lambda r}-1|\les |\lambda| r$, we see that this piece may be bounded in modulus by $|\lambda| r^{-1}$, as can the contribution of $\lambda \mG_1$.
	For derivatives, we have that differentiation is bounded by multiplication by $r$, so that
	\begin{align}\label{eq:resolv high derivs}
		|\p_\lambda^k\mR_0^\pm(\lambda)(x,y)|
		&\les \bigg(\frac{|\lambda|}{r} +\frac{1}{r^2}\bigg) r^k
		\les  |\lambda| r^{k-1}+r^{k-2}.
	\end{align}
	Now explicitly writing $\mE_1^\pm(\lambda,r)$, we have
	$      |\mE_1^\pm(\lambda, r)| =  |\mR_0^\pm(\lambda,x,y)-\mG_0(x,y)-\lambda \mG_1(x,y) | \les |\lambda|r^{-1}$.
	Since $\mG_0$ is $\lambda$ independent, derivatives are controlled by \eqref{eq:resolv high derivs}.  So that, when $|\lambda|r\gtrsim 1$ we have
	$$
	|\p_\lambda^k\mE_1^\pm(\lambda, r)| \les |\lambda|r^{k-1}+r^{k-2}, \qquad k=0,1,2.
	$$
	The bounds on $|\lambda| r\gtrsim 1$ may be freely multiplied by positive powers of $|\lambda|r$, while the bounds in \eqref{eqn:otherE0} my be divided by powers of $|\lambda |r$ to obtain the bounds in the claim.
	
	The argument for $\mathcal E_2^\pm(\lambda,x,y)$ proceeds similarly noting that 
	$\mathcal E_2^\pm(\lambda,x,y):=\mathcal E_1^\pm(\lambda,x,y)-i\lambda^2 \mG_2^\pm(x,y)$.  When $|\lambda|r<1$, the Taylor expansion gives an error of size $|\lambda|^3 r$, which may be differentiated freely.  When $|\lambda |r>1$, we note that
	$$
	|\partial_\lambda^k (i\lambda^2 \mG_2^\pm(x,y))| \les \lambda^{2-k}, \qquad k = 0, 1, 2.
	$$
	Combining this with the estimates for $\mathcal E_1^\pm(\lambda,r)$ suffices to prove the claim.
\end{proof}

These bounds may be used to obtain Lipschitz bounds on the error term.

\begin{lemma}\label{lem:R0 error Lip}
	Let $|\lambda_1|\leq |\lambda_2| \leq 1$, then we have the following Lipschitz bounds for the error terms.  For any choice of $0\leq \gamma \leq 1$ we have:
	\begin{align*}
		|\mathcal{E}_1^\pm(\lambda_2,|x-y|)-\mathcal{E}_1^\pm(\lambda_1,|x-y|)|
		& \les |\lambda_2-\lambda_1|^{\gamma}\bigg(\frac{|\lambda_2|}{|x-y|^{1-\gamma}}\bigg), \\
		|\p_\lambda \mathcal{E}_1^\pm(\lambda_2,|x-y|)-\p_\lambda \mathcal{E}_1^\pm(\lambda_1,|x - y|)|
		& \les |\lambda_2-\lambda_1|^{\gamma}|\lambda_2|^{1-\gamma}(1 + |x-y|^{\gamma}).
	\end{align*}
	Similarly, for any $0\leq \ell \leq 1$
	\begin{align*}
		|\mathcal{E}_2^\pm(\lambda_2,|x-y|)-\mathcal{E}_2^\pm(\lambda_1,|x-y|)| 
		& \les |\lambda_2-\lambda_1|^{\gamma} |\lambda_2|^{1+\gamma+\ell}|x-y|^\ell, \\
		|\p_\lambda \mathcal{E}_2^\pm(\lambda_2,|x-y|)-\p_\lambda \mathcal{E}_2^\pm(\lambda_1,|x-y|)| 
		& \les|\lambda_2-\lambda_1|^{\gamma}|\lambda_2|^{(1+\ell)(1-\gamma)}(1+|x-y|^{\gamma+\ell(1-\gamma)}).
	\end{align*}
\end{lemma}

\begin{proof}
	We prove the claim for $\mathcal E_1^\pm$, the proof for $\mathcal E_2^\pm$ is similar, but simpler.  The proof is independent of the $\pm$, so we drop the superscript.
	We begin by bounding $|\mathcal{E}_1(\lambda_1,|x-y|)-\mathcal{E}_1(\lambda_2,|x-y|)|$. We use the triangle inequality and Lemma~\ref{lem:R0 expansions} to obtain
	\begin{equation*}
		|\mathcal{E}_1(\lambda_1, |x-y|) - \mathcal{E}_1(\lambda_2, |x-y|)|
		\les \frac{|\lambda_2|}{|x-y|},
	\end{equation*}
	Using the mean value theorem, we obtain the alternative bound
	\begin{align*}
		|\mathcal{E}_1(\lambda_1,|x-y|)-\mathcal{E}_1(\lambda_2,|x-y|)| 
		& \les |\lambda_2-\lambda_1|\sup_{\lambda \in I} \left|\p_\lambda \mathcal{E}_1(\lambda,|x-y|)\right| 
		\les |\lambda_2-\lambda_1||\lambda_2|.
	\end{align*}
	Here $I$ is the interval $[\min(\lambda_1,\lambda_2), \max(\lambda_1,\lambda_2)]$.
	Via interpolation, we obtain
	\begin{equation*}
		|\mathcal{E}_1(\lambda_1,|x-y|)-\mathcal{E}_1(\lambda_2,|x-y|)|\les |\lambda_2-\lambda_1|^{\gamma}\frac{|\lambda_2|}{|x-y|^{1-\gamma}} .
	\end{equation*}
	We use the same approach for $\p_{\lambda} \mathcal{E}_1(\lambda,|x-y|)$. These give us the two bounds
	\begin{align*}
		|\p_\lambda \mathcal{E}_1(\lambda_1,|x-y|) - \p_\lambda \mathcal{E}_1(\lambda_2,|x-y|)| & \les |\lambda_2| \\
		|\p_\lambda \mathcal{E}_1(\lambda_1,|x-y|)-\p_\lambda \mathcal{E}_1(\lambda_2,|x-y|)| & \les  |\lambda_2-\lambda_1|\sup_{\lambda \in I} \left|\p_\lambda^2 \mathcal{E}_1(\lambda,|x-y|)\right|\\
		&\les |\lambda_2-\lambda_1|(1+|\lambda_2||x-y|).
	\end{align*}
	We can then interpolate between these bounds to obtain
	\begin{align*}
		|\p_\lambda \mathcal{E}_1(\lambda_1,x,y)-\p \lambda \mathcal{E}_1(\lambda_2,x,y)| & \les |\lambda_2-\lambda_1|^{\gamma}(1+|\lambda_2||x-y|)^{\gamma}|\lambda_2|^{1-\gamma}\\
		& \les |\lambda_2-\lambda_1|^{\gamma}(|\lambda_2|^{1-\gamma}+|\lambda_2||x-y|^{\gamma}). %\tag*{\qedsymbol}
	\end{align*}%\renewcommand{\qedsymbol}{}
\end{proof}

Combining Lemma~\ref{lem:R0 error Lip} with the first expansion in Lemma~\ref{lem:R0 expansions}, since the first term is independent of $\lambda$ we have
$$
\mR_0^\pm(\lambda_2)(x,y)-\mR_0^\pm(\lambda_1)(x,y)=(\lambda_2-\lambda_1)\mG_1+\mathcal{E}_1^\pm (\lambda_2,|x-y|)-\mathcal{E}_1^\pm(\lambda_1,|x-y|).
$$
We note that the $\mG_1$ terms cancel in the difference of the derivatives.  From this we can see
\begin{corollary}\label{cor:Lips R0}
	Lipschitz bounds for $\mathcal E_1^\pm(\lambda,|x-y|)$  also hold for the Dirac resolvent.  Namely, for $|\lambda_1|\leq |\lambda_2|\leq 1$ we have
	\begin{align*}
		|\mR_0^\pm(\lambda_2)(x,y)-\mR_0^\pm(\lambda_1)(x,y)| 
        & \les \frac{|\lambda_2-\lambda_1|}{|x-y|}+|\lambda_2-\lambda_1|^{\gamma}\bigg(\frac{|\lambda_2|}{|x-y|^{1-\gamma}}\bigg).  \\
		|\p_\lambda \mathcal{R}^\pm_0(\lambda_2)(x,y)-\p_\lambda \mathcal{R}^\pm_0(\lambda_1)(x,y)| 
        & \les|\lambda_2-\lambda_1|^{\gamma}|\lambda_2|^{1-\gamma}(1+|x-y|^{\gamma}).
	\end{align*}
\end{corollary}

To control the evolution of the perturbed solution operator, we must distinguish between the cases when zero energy is regular and exceptional.  To that end, we make the following definition, which is similar to the definitions for the massive cases \cite{egd,EGT,EGT2,egd1} and massless case \cite{EGG2} respectively.  These had roots in earlier works on the Schr\"odinger operators such as \cite{ES,eg2,eg3}.
\begin{defin}\label{def:resonances}
	
	We make the following definitions that characterize zero energy obstructions.
	\begin{enumerate}[i)]
		\item We define zero energy to be regular if $T_0=M^\pm (0)$ is invertible on $L^2(\R^3)$.  
		
		\item 	We say that zero is not regular   if $T_0$ is not invertible on $L^2(\R^3)$. In this case we define $S_1$ is the Riesz projection onto the kernel of $T_0$, so that $T_0+S_1$ is invertible on $L^2(\R^3)$.  We show, in Section~\ref{sec:thresholds}, the connection of zero energy not being regular  to the existence of zero energy eigenvalues of $\mathcal H=D_0+V$.
		
		\item Noting that $v\mG_{0}v^*$ is compact and self-adjoint, it follows that  $T_0=U+v\mG_{0}v^*$ is  a compact perturbation of $U$.  Since the spectrum of $U$ is in $\{\pm 1\}$, zero is an isolated point of the spectrum of $T_0$ and the kernel is a finite-dimensional subspace of $L^2(\R^3)$.  It then follows that $S_1$ is a finite rank projection.
		
	\end{enumerate}
	
\end{defin}

We employ the following terminology from \cite{Sc2,eg2,eg3}:
\begin{defin}
	We say an operator $T \colon L^2(\R^3) \to L^2(\R^3)$ with kernel
	$T(\,\cdot\,,\,\cdot\,)$ is absolutely bounded if the operator with kernel
	$|T(\,\cdot\,,\,\cdot\,)|$ is bounded from $L^2(\R^3)$ to $ L^2(\R^3)$. 	
\end{defin}
Recall the Hilbert-Schmidt norm of an operator $T$ with integral kernel $T(x,y)$ is defined by
$$
\|T\|_{HS}^2=\int_{\mathbb R^6} |T(x,y)|^2\, dx\, dy.
$$
Recall that Hilbert-Schmidt and finite-rank operators are absolutely bounded.

We now proceed to building expansions for the operators $M^\pm(\lambda)$ when zero energy is regular.  The following estimate is used frequently.

\begin{lemma}\label{lem:spatial estimates}
	Fix $x,y\in \R^n$, with $0 \leq k, \ell < n$, $\delta > 0$, $k + \ell + \delta \geq n$, $k + \ell \neq n$:
	\begin{align*}%\label{eqn:R3kell}
		\int_{\R^n} \frac{\ang{z}^{-\delta-}}{|z-x|^k |y-z|^\ell} dz 
		\lesssim
		\begin{cases}
			|x - y|^{-\max \set{0, k + \ell - n}}                & \t{if } |x-y|   <  1; \\
			|x - y|^{-\min \set{k, \ell, k + \ell + \delta - n}} & \t{if } |x-y| \geq 1.
		\end{cases} 
	\end{align*}
	Consequently, we have that
	\begin{align}\label{eqn:R3kell}
		\int_{\R^n} \frac{\ang{z}^{-\delta-}}{|z-x|^k |y-z|^\ell} dz 
		\lesssim \frac{1}{|x-y|^p}
	\end{align}
	where we may take any $p\in [\max \{0, k + \ell - n\},\min \{k, \ell, k + \ell + \delta - n\}]$ as desired. 
\end{lemma}
Note that for $a, b, \varepsilon > 0$ and $\ell > k$, we have $a^{-k}b^{-\ell} \lesssim a^{-k}(b^{-\ell-\varepsilon} + b^{-\ell+\varepsilon})$, which allows for the application of Lemma \ref{lem:spatial estimates} when $k + \ell=n = 3$.
\begin{proof}
	The first claim is Lemma~6.3 in \cite{EG1}. The second claim follows by noting that $\min \{k, \ell, k + \ell + \delta - n\}\geq \max \{0, k + \ell - n\}$ noting that if $|x-y|<1$ selecting $p\geq \max \{0, k + \ell - n\}$ increases the upper bound, while if $|x-y|>1$ selecting $p\leq \min \{k, \ell, k + \ell + \delta - n\}$ also increases the upper bound.
\end{proof}

To handle the case when either $k$ or $\ell = 0$, we recall Lemma~3.8 in \cite{goldVis}, which we state in less generality:
\begin{lemma}\label{lem:GolVis}
	
	Suppose that $0\leq \delta,k<3$ with $k+\delta>3$, then
	$$
		\int_{\R^3}\frac{\la x\ra^{-\delta}}{|x-y|^k}\, dx \les \la y\ra^{3-k-\delta}.
	$$
	
\end{lemma}

\begin{lemma}\label{lem:M exp}
	Assume that $|V(x)|\les \la x\ra^{-\delta}$ for some $\delta>0$.
	We have the expansion
	\begin{align*}
		M^\pm(\lambda) = T_0+\lambda v\mG_1v^*+M_0^\pm(\lambda),
	\end{align*}
	where, for $\lambda$ in a neighborhood of zero, we have ${\| M_0(\lambda)\|}_{HS}<\infty$ provided $\delta>1$, ${\|\partial_\lambda M_0(\lambda)\|}_{HS}\les |\lambda|<\infty$ provided $\delta>3$, and ${\|\partial_\lambda^2 M_0(\lambda)\|}_{HS}<\infty$ provided $\delta > 5$. Furthermore, for any choice of $0\leq \gamma \leq 1$ and for $|\lambda_1|\leq |\lambda_2|\leq 1$, we have
	\begin{align*}
		{\|M_0^\pm(\lambda_2) -M_0^\pm(\lambda_1)\|}_{HS} \les |\lambda_2-\lambda_1|^\gamma,
		\qquad \text{provided } \delta>1+2\gamma.
	\end{align*}
	Furthermore, if $\delta > 3 + 2\gamma$ we have
	\begin{align*}
		{\|\p_\lambda M_0^\pm(\lambda_2) -\p_\lambda M_0^\pm(\lambda_1)\|}_{HS} \les |\lambda_2-\lambda_1|^\gamma.
	\end{align*}
	The Lipschitz bounds also hold for $M^\pm(\lambda)$ in place of $M_0^\pm(\lambda)$.  
\end{lemma}

\begin{proof}
	Recall the definition of $M^\pm$ as well as $\mR_0^\pm$
	\begin{align*}
		M^\pm (\lambda)
		= U + v \mR_0^\pm (\lambda) v^* 
		= U + v \mG_0 v^*+\lambda v\mG_1 v^* + v \mE_1^\pm(\lambda) v^*
		= U + v \mG_0 v^* + M_0^\pm(\lambda).
	\end{align*}
	To compute the Hilbert-Schmidt norms, using \eqref{eqn:Gj def} and Lemma~\ref{lem:R0 expansions}, we have
	\begin{align*}
		{\left\| M_0^\pm(\lambda)\right\|}_{HS}^2
		& = \int_{\R^6} |v(x) \mE_1^\pm(\lambda,x,y)v^*(y)|^2\, dx\, dy  \lesssim |\lambda|^2 \int_{\R^6}  \frac{\ang{x}^{-\delta} \ang{y}^{-\delta}}{|x-y|^2}   \, dx\, dy \les |\lambda|^2.
	\end{align*}
	Lemma~\ref{lem:spatial estimates} was used in the $x$ integral, we use \eqref{eqn:R3kell} with $p=2$. Another application in the $y$ integral shows the integral is bounded.
	Then, applying Lemmas~\ref{lem:R0 expansions} and \ref{lem:spatial estimates} show the bound. Similar computations show ${\left\| \p_\lambda M_0^\pm(\lambda)\right\|}_{HS} \les |\lambda|$ and ${\left\| \p_\lambda^2 M_0^\pm(\lambda)\right\|}_{HS} < \infty$ requiring more decay from the potential to ensure that
	\begin{align*}
		\int_{\R^6} \ang{x}^{-\delta} \ang{y}^{-\delta} dx\, dy \qquad \text{and} \qquad 	\int_{\R^6}  \ang{x}^{-\delta}|x-y|^2 \ang{y}^{-\delta} dx\, dy
	\end{align*}
	converge.  Applying Lemma~\ref{lem:spatial estimates} requires $\delta>3$ and $\delta>5$, respectively.
	Now, we consider the first Lipschitz bound.  For uniformity of presentation, we use that $|\lambda|\les 1$ to obtain the bounds in the statement. By Lemma~\ref{lem:R0 error Lip}, we have
	\begin{align*}
		|M_0^\pm(\lambda_2)-M_0^\pm(\lambda_1)|\les |\lambda_2-\lambda_1|^{\gamma} |v(x)|\bigg(\frac{|\lambda_2|}{|x-y|^{1-\gamma}}\bigg) |v^*(y)|.
	\end{align*}
	The Hilbert-Schmidt norm is bounded by Lemma~\ref{lem:spatial estimates}, provided $\delta > 1 + 2\gamma$. 
	
	Finally for the derivative, again by Lemma~\ref{lem:R0 error Lip}, we have
	\begin{align*}
		|\p_\lambda M_0^\pm(\lambda_2) - \p_\lambda M_0^\pm(\lambda_1)| \les  |v(x)| |\lambda_2-\lambda_1|^{\gamma}|\lambda_2|^{1-\gamma}(1+|x-y|^{\gamma})|v^*(y)|
	\end{align*}
	This is Hilbert-Schmidt provided that $\delta>3+2\gamma$.
	
	Finally, note that $M^{\pm}(\lambda_2) - M^{\pm}(\lambda_1) = (\lambda_2-\lambda_1)v\mG_1v^*+ M_0^{\pm}(\lambda_2) - M_0^{\pm}(\lambda_1)$ since $T_0$ has no $\lambda$ dependence,
	while  $\p_\lambda M^{\pm}(\lambda_2) -\p_\lambda M^{\pm}(\lambda_1)=\p_\lambda \mE_1^\pm(\lambda_2)-\p_\lambda \mE_1^\pm(\lambda_1)$ since $v\mG_1v^*$ is independent of $\lambda$.  Using that $|\lambda_2-\lambda_1|\les |\lambda_2|\les 1$, and the argument above for $M_0^\pm(\lambda)$ suffices to prove the claimed bounds for $M^\pm(\lambda)$.
\end{proof}

\begin{lemma}\label{lem:Minv zeroreg}
	If zero is a regular point of the spectrum and $|V(x)|\les \la x\ra^{-\delta}$ for some $\delta>1$.  Then $M^{\pm}(\lambda)$ is an invertible operator with uniformly bounded inverse on a sufficiently small interval $0<|\lambda| \ll 1$.    Furthermore,
	\begin{enumerate}[i)]
		\item If $\delta>1+2\gamma$ for some $0\leq \gamma \leq 1$, then for $0<|\lambda_1|\leq |\lambda_2|\ll1$, we have
		\begin{align*}%\label{eqn:minvlipbound}
			{\| (M^\pm)^{-1}(\lambda_2)-(M^\pm)^{-1}(\lambda_1)\|}_{HS}\les |\lambda_2-\lambda_1|^\gamma.
		\end{align*}
		
		\item If $\delta > 3$, then 
		\begin{align}\label{eqn:minvderbound}
			{\| \partial_\lambda (M^\pm)^{-1}(\lambda) \|}_{HS}<\infty.
		\end{align}
		
		\item If $\delta>3+2\gamma$ for some $0\leq \gamma \leq 1$, then for $0<|\lambda_1|\leq |\lambda_2|\ll1$, we have
		\begin{align*}%\label{eqn:minvderlipbound}
			{\| \partial_\lambda (M^\pm)^{-1}(\lambda_2)-\partial_\lambda(M^\pm)^{-1}(\lambda_1)\|}_{HS}\les |\lambda_2-\lambda_1|^\gamma |\lambda_2|.
		\end{align*}
	\end{enumerate}
\end{lemma}

\begin{proof}
	We consider the $+$ case only and drop the superscript.
	By a Neumann series expansion, if $T_0$ is invertible on $L^2$, denote the inverse by $D_1:=T_0^{-1}$ as an operator on $L^2$.  
	The fact that $D_1$ is an absolutely bounded operator is established in Lemma~2.10 of \cite{EGT}, which considered the massive operator.  That proof applies here with only minimal modifications.
	
	By Lemma~\ref{lem:M exp} and a Neumann series expansion,
	\begin{align*}
		M^{-1}(\lambda) = (T_0+\lambda v\mG_1v^*+ M_0(\lambda))^{-1}=D_1(\mathbbm 1+\lambda v\mG_1v^*D_1+M_0(\lambda)D_1)^{-1}=D_1 +O(|\lambda|),
	\end{align*}
	where the error term is understood as a Hilbert-Schmidt operator. In particular, $M^{-1}(\lambda)$ is an absolutely bounded operator on $L^2(\R^3)$.
	
	By the resolvent identity $M^{-1}(\lambda_1)-M^{-1}(\lambda_2)=M^{-1}(\lambda_1)(M(\lambda_1)-M(\lambda_2))M^{-1}(\lambda_2)$. Since $M$ was shown to be invertible, we may apply Lemma \ref{lem:M exp} to see
	\begin{align*}
		\left\|M^{-1}(\lambda_1)-M^{-1}(\lambda_2)\right\|_{HS} 
		& \lesssim \|M^{-1}(\lambda_1)\|_{L^2 \to L^2}\|(M(\lambda_1)-M(\lambda_2))\|_{HS}\|M^{-1}(\lambda_2)\|_{L^2 \to L^2}\\
		& \lesssim |\lambda_2-\lambda_1|^{\gamma},
	\end{align*}
	provided $1+2\gamma>\delta$. For the second claim, \eqref{eqn:minvderbound}, we use the  identity
	\begin{equation}\label{eqn:deriv diff}
		\p_\lambda M^{-1}(\lambda)=-M^{-1}(\lambda)(\p_\lambda M(\lambda))M^{-1}(\lambda)
	\end{equation}
	By Lemma \ref{lem:M exp} and the boundedness of $M^{-1}$ established above, we see
	\begin{equation*}
		\|\p_\lambda M^{-1}(\lambda)\|_{HS}\lesssim \|\p_\lambda M(\lambda)\|_{HS}\lesssim \lambda < \infty.
	\end{equation*}
	Finally, for the Lipschitz bound we recall the following useful algebraic identity 
	\begin{align}\label{eqn:alg identityl1l2}
		\prod_{k = 0}^M A_k(\lambda_2) - \prod_{k = 0}^M A_k (\lambda_1)
		= \sum_{\ell=0}^M \bigg(\prod_{k = 0}^{\ell - 1}A_k (\lambda_1)\!\bigg)
		\big(A_\ell (\lambda_2) - A_\ell (\lambda_1)\big)\bigg(
		\prod_{k = \ell + 1}^M A_k(\lambda_2)\!\bigg).
	\end{align}
	The algebraic identity ensures that there is a difference of the same operators evaluated at $\lambda_2$ and $\lambda_1$ on which we may apply the previously obtained Lipschitz bounds. By \eqref{eqn:alg identityl1l2} and \eqref{eqn:deriv diff}
	\begin{multline*}
		\p_\lambda M^{-1}(\lambda_2)-\p_\lambda M^{-1}(\lambda_1)
		=[M^{-1}(\lambda_2)-M^{-1}(\lambda_1)] \p_\lambda(M(\lambda_2))M^{-1}(\lambda_2)\\
		+M^{-1}(\lambda_1)[\p_\lambda M^{-1}(\lambda_2)-\p_\lambda M^{-1}(\lambda_1)]M^{-1}(\lambda_2)
		+M^{-1}(\lambda_1)\p_\lambda(M(\lambda_1))[M^{-1}(\lambda_2)-M^{-1}(\lambda_1)].
	\end{multline*}
	Since $\p_\lambda M^{-1}(\lambda)$ and $M^{-1}(\lambda)$ are absolutely bounded, by Lemma~\ref{lem:M exp}, we see 
	\begin{align*}
		\|\p_\lambda M^{-1}(\lambda_1)-\p_\lambda M^{-1}(\lambda_2)\|_{HS}
		& \lesssim |\lambda_1-\lambda_2|^\gamma |\lambda_2|.
	\end{align*}
	Here, we need $\delta>3 + 2\gamma$ to apply the Lipschitz bound of Lemma~\ref{lem:M exp}.
\end{proof}

\section{Dispersive bounds when zero is regular}\label{sec:zero reg}

Now that we have developed appropriate expansions for the free and perturbed resolvent operators, we are ready to prove the low energy claims in Theorem~\ref{thm:main} in the case when zero is regular.  Using the differentiability properties established in Section~\ref{sec:low expansions}, we develop expansions of the spectral measure in a neighborhood of zero with similar properties.  Using the Stone's formula, \eqref{eq:Stone}, we reduce the evolution bounds to oscillatory integrals that we control.  We first prove the uniform, $L^1\to L^\infty$, bounds.  Then, in Subsection~\ref{subsec:wtd bds} we utilize the more delicate Lipschitz continuity bounds to prove the large time-integrable bounds at the cost of mapping between weighted spaces.  Finally, we apply the Lipschitz bounds to prove the low energy version of Theorem~\ref{thm:weak decay}.

By iterating the symmetric resolvent identity, \eqref{eq:pert1}, we obtain the Born series expansion:
\begin{multline}\label{eqn:born series}
	\mR_V^\pm(\lambda)= \mR_0^\pm(\lambda)- \mR_0^\pm(\lambda) V \mR_0^\pm(\lambda)+ \mR_0^\pm (\lambda)V \mR_0^\pm(\lambda) V  \mR_0^\pm(\lambda)\\
	-\mR_0^\pm(\lambda) V \mR_0^\pm(\lambda) v^*(M^{\pm})^{-1}  v \mR^{\pm}_0 (\lambda) V  \mR_0^\pm(\lambda).
\end{multline}
We iterate so that we have two resolvents on either side of $(M^\pm)^{-1}$ since the kernel of $\mR_0^\pm(\lambda)$ has a leading term that is not locally $L^2(\R^3)$.  The main goal of this section is to use the Born series expansion to prove low energy dispersive bounds.

\begin{prop}\label{prop:low time bound}
	Assume that $|V(x)|\les \la x\ra^{-\delta}$ for some $\delta>3$. Then
	\begin{align*}
		\sup_{x,y\in \R^3} \left| \int_{\R} e^{-it\lambda} \chi(\lambda) (\mR_V^+-\mR_V^-)(\lambda)(x,y)\, d\lambda\right| \les \la t\ra^{-1}.
	\end{align*}
\end{prop}

We prove this Proposition through a series of Lemmas.  First, we control the contribution of the Born series, the first three terms in \eqref{eqn:born series}, to the Stone's formula.  We note that the first term is the free evolution and is controlled in Theorem~\ref{thm:free}.

\begin{lemma}\label{lem:born series}
	Assume $|V(x)|\les \la x\ra^{-\delta}$ for $\delta>2$. Then for any fixed $k \in \mathbb N$, we have the bound
	\begin{align*}
		\sup_{x,y\in\R^3}\left|\int_\R e^{-it\lambda} \chi(\lambda)\!\left(\mc{R}_0^+(\lambda) (V \mc{R}_0^+(\lambda) )^k - \mc{R}_0^-(\lambda) (V\mc{R}_0^-(\lambda) )^k\right) \!d\lambda\right| 
		\lesssim \frac{1}{|t|}.    
	\end{align*}
\end{lemma}
\begin{proof}
	To utilize the difference between the `$+$' and `$-$' resolvent operators, we use the following algebraic identity.
	\begin{align}\label{eqn:alg identity}
		\prod_{k = 0}^M A_k^+ - \prod_{k = 0}^M A_k^-
		= \sum_{\ell=0}^M \bigg(\prod_{k = 0}^{\ell - 1}A_k^-\bigg)
		\big(A_\ell^+ - A_\ell^-\big)\bigg(
		\prod_{k = \ell + 1}^M A_k^+\bigg).
	\end{align}
	We first prove the claim when $k=1$. Integrating by parts and expanding the derivative for the $k = 1$ case yields four terms, the first of which we will bound since the other three follow similarly. Recall that $\mu_0(\lambda, x,y) = \chi(\lambda)[\mR_0^+-\mR_0^-](\lambda, x, y)$. Using \eqref{eqn:R0 bounds} and \eqref{eqn:mu0 bounds} we integrate by parts once to see
	\begin{multline*}
		\bigg|\int_\R \frac{e^{-it\lambda}}{t} \int_{\R^3}\p_\lambda\big(\mu_0(\lambda)( x, z)V(z) \mc{R}_0^\pm(\lambda, z, y) \big) dzd\lambda\bigg|\\
		\qquad \qquad \lesssim \frac{1}{|t|}\int_{-1}^1 \int_{\R^3}   \la z\ra^{-\delta} \left(\frac{1}{|z - y|^2} + \frac{1}{|z - y|}+\frac{1}{|x - z||z-y|}+\frac{1}{|z - y|}\right)  \!dz\,d\lambda 
		\lesssim  \frac{1}{|t|},
	\end{multline*}
	here we  use Lemma~\ref{lem:spatial estimates}  since $\delta > 2$.  This bound is uniform in $x$ and $y$. The boundedness of the integral for small $t$ follows without integrating by parts but similarly using \eqref{eqn:mu0 bounds}, \eqref{eqn:R0 bounds} and Lemma~\ref{lem:spatial estimates}, to control the spatial integrals.
	
	When $k>1$ we note that applying Lemma~\ref{lem:spatial estimates} with $\delta >2$, and \eqref{eqn:R0 bounds},  on the support of $\chi(\lambda)$ we have
	\begin{align*}
		\bigg| \int_{\R^3}\mR_0^\pm(\lambda)(x,z_1) V(z_1) \mR_0^\pm(z_1,z_2)\, dz_1  \bigg| \les \frac{1}{|x-z_2|^2}+\frac{1}{|x-z_2|},
	\end{align*}
	which is the same upper bound we use in these arguments for $\mR_0^\pm (\lambda)(x,z_2)$.  Hence one may reduce to the argument when $k=1$ by first integrating in the spatial variables of resolvents that are not differentiated to reduce to bounding spatial integrals of the form considered above.
\end{proof}
\begin{lemma}\label{prop:perturbed bound}
	If $|V(x)| \les \la x\ra^{-\delta}$ for any $\delta > 3$ and $\Gamma(\lambda)$ is a absolutely bounded operator satisfying
	$$
	\big\| | \Gamma(\lambda) | \big\|_{L^2 \to L^2}+\big\| | \lambda \partial_\lambda \Gamma(\lambda) | \big\|_{L^2 \to L^2}\les |\lambda|^{0+},
	$$ 
	then we have the bound
	\begin{align*}
		\sup_{x,y\in \R^3} \bigg| \int_{\R} e^{-it\lambda} \chi(\lambda) \mR_0^\pm(\lambda) V \mR_0^\pm(\lambda) v^* \Gamma(\lambda)  v \mR^{\pm}_0 (\lambda) V  \mR_0^\pm(\lambda)\, d\lambda \bigg| \les \la t\ra^{-1}.
	\end{align*}
\end{lemma}
The assumptions on the operator $\Gamma(\lambda)$ are far less stringent than needed here since $\partial_\lambda (M^{\pm})^{-1}(\lambda)$ is bounded in a neighborhood of zero when zero is regular.  We choose to prove a more general result to reuse in the analysis when zero is not regular.
\begin{proof}
	As before we integrate by parts once to bound
	\begin{align}\label{eqn:BS series tail}
		\frac{1}{|t|}\sup_{x,y\in \R^3} \left|\int_{\R} e^{-it\lambda} \partial_\lambda \left(\chi(\lambda) \mR_0^\pm(\lambda) V \mR_0^\pm(\lambda) v^* \Gamma(\lambda)  v \mR^{\pm}_0 (\lambda) V  \mR_0^\pm(\lambda)\right) d\lambda \right|.
	\end{align}
	The assumptions on $\Gamma(\lambda)$ and \eqref{eqn:R0 bounds}  ensure there are no boundary terms at zero.  Furthermore, when $|\lambda|\ll 1$, by \eqref{eqn:R0 bounds} we have
	\begin{multline*}
		\left| (\mR_0^\pm V\mR_0^\pm)(\lambda)(z_2,x) \right|
		\les \int_{\R^3}\ang{z_1}^{-\delta}\left(\frac{1}{\abs{z_2-z_1}^2\abs{z_1-x}^2}+\frac{1}{\abs{z_2-z_1}\abs{z_1-x}}\right)\, dz_1 \les 
		\frac{1}{|z_2-x|},
	\end{multline*}
	by Lemma~\ref{lem:spatial estimates} provided $\delta>2$. Using Lemma~\ref{lem:spatial estimates} again in the $z_2$ integral we see that
	\begin{align}\label{eqn:4.3a}
		\sup_{x \in \R^3}\left\| (v\mR_0^\pm V\mR_0^\pm)(\lambda)(\, \cdot \,, x) \right\|^2_{L^2} 
		&
		\les 1.  
	\end{align}
	By duality, $\left\|\big(\mR_0^\pm V\mR_0^\pm v^*\big)(\lambda)(x, \,\cdot\,)\right\|_{L^2} \les 1$ holds uniformly in $x$ as well.
	Using \eqref{eqn:R0 bounds} and Lemma~\ref{lem:spatial estimates}, we see 
	\begin{align}\label{eqn:4.3b}
		\sup_{x \in \R^3}\left\|\partial_\lambda\big(v\mR_0^\pm V\mR_0^\pm\big)(\lambda)(x,\,\cdot\,)\right\|^2_{L^2}  
		\les 1 .  
	\end{align}
	One uses that $\delta>3$ here to ensure that the contribution the slowest decaying terms after the applying Lemma~\ref{lem:spatial estimates} are in $L^2$.
	Bringing everything together, we may express the contribution of \eqref{eqn:BS series tail} by rewriting the integral as
	$$
	\eqref{eqn:BS series tail}=\frac{1}{t} \int_{\R} e^{-it\lambda}\mE(\lambda)\,d\lambda,
	$$
	where 
	\begin{align*}
		\mE(\lambda) = \sum \big(\p_\lambda^{k_1}\chi(\lambda)\big) \p_\lambda^{k_2}\big(\mR_0^\pm(\lambda) V \mR_0^\pm(\lambda) v^* \big)\big(\p_\lambda^{k_3}\Gamma(\lambda)\big) \p_\lambda^{k_4}\big(v \mR^{\pm}_0 (\lambda) V \mR_0^\pm(\lambda)\big),
	\end{align*}
	and the sum is taken over the indices $k_j \in \set{0,1}$ subject to $k_1 + k_2 + k_3 + k_4 = 1$. By the absolute boundedness of $\Gamma(\lambda)$, we have
	\begin{align*}
		|\mE(\lambda)| 
		& = \big|\big(\p_\lambda^{k_1}\chi(\lambda)\big) \p_\lambda^{k_2}\big(\mR_0^\pm(\lambda) V \mR_0^\pm(\lambda) v^* \big)\big(\p_\lambda^{k_3}\Gamma(\lambda)\big) \p_\lambda^{k_4}\big(v \mR^{\pm}_0 (\lambda) V \mR_0^\pm(\lambda)\big)\big| \\
		& = \ang{\p_\lambda^{k_2}\big(v\mR_0^\mp(\lambda) V \mR_0^\mp(\lambda)\big), \big(\p_\lambda^{k_3}\Gamma(\lambda)\big) \p_\lambda^{k_4}\big(v \mR^{\pm}_0 (\lambda) V   \mR_0^\pm(\lambda)\big)}_{L^2} \\
		& \lesssim \big\|\p_\lambda^{k_2}\big((v\mR_0^\pm V\mR_0^\pm)(\, \cdot \,, x)\big)\big\|_{L^2}\big\||\p_\lambda^{k_3}\Gamma(\lambda, z_2, \,\cdot\,)| \p_\lambda^{k_4} \big((\mR_0^\pm V\mR_0^\pm v^*)(\,\cdot\,, y)\big)\big\|_{L^2} \\
		& \lesssim \big\|\p_\lambda^{k_2}\big((v\mR_0^\pm V\mR_0^\pm)(\, \cdot \,, x)\big)\big\|_{L^2}\big\||\p_\lambda^{k_3}\Gamma(\lambda)|\big\|_{L^2 \to L^2}\big\|\p_\lambda^{k_4} \big((\mR_0^\pm V\mR_0^\pm v^*)(\,\cdot\,, y)\big)\big\|_{L^2}
		\lesssim |\lambda|^{-1+},
	\end{align*}
	which holds over the support of $\chi$, which is contained in the interval $[-1,1]$. Applying this bound to \eqref{eqn:BS series tail}, we have
	\begin{align*}
		\eqref{eqn:BS series tail}
		\lesssim \frac{1}{|t|}\sup_{x,y \in \R^3}\left|\int_{\R} e^{-it\lambda}\mE(\lambda)\,d\lambda\right|
		\lesssim \frac{1}{|t|}\sup_{x,y \in \R^3}\int_{-1}^1  |\lambda|^{-1+} \,d\lambda
		\lesssim \frac{1}{|t|}
	\end{align*}
	as desired.  The claim for boundedness for small $|t|$ follows the argument for when $k_1=1$ without integrating by parts.
\end{proof}
Now, we prove Proposition~\ref{prop:low time bound}.
\begin{proof}[Proof of Proposition \ref{prop:low time bound}]
	By expanding $\mR_V^\pm$ into a Born series expansion as in \eqref{eqn:born series}, we can control the contribution of each term.  The contribution of first term in \eqref{eqn:born series} to \eqref{eqn:stone} is controlled by Theorem~\ref{thm:free}, the contribution of the second and third are controlled by Lemma~\ref{lem:born series}.  For the final term, we do not utilize the difference between the `+' and `$-$' resolvent, but control each by applying Lemma~\ref{prop:perturbed bound}.
\end{proof}

\subsection{Weighted dispersive bounds when zero is regular}\label{subsec:wtd bds}

We now turn to showing that the large time integrable bounds hold when zero is regular.  Here we utilize the Lipschitz continuity of the perturbed resolvent and its first derivative in a neighborhood of the threshold at $\lambda=0$.  Our main goal is to show the bound
\begin{prop}\label{prop:low wtd}
	Fix $0<\gamma\leq 1$.  If $|V(x)|\les \la x\ra^{-\delta}$ for some $\delta>3+2\gamma$, then for $|t|>1$ we have the weighted bound
	\begin{align*}
		\bigg| \int_{\R} e^{-it\lambda} \chi(\lambda)  [\mR_V^+-\mR_V^-](\lambda)(x,y) \, d\lambda \bigg| \les  \frac{\la x\ra^\gamma \la y \ra^\gamma}{|t|^{1+\gamma}}.
	\end{align*}
\end{prop}

This tells us that the low energy portion of the evolution satisfies a large time-integrable bound as an operator from $L^{1,\gamma}\to L^{\infty,-\gamma}$.  As in the proof of the uniform bound in the previous subsection, we consider the Born series expansion \eqref{eqn:born series} and bound each term individually.  
\begin{lemma}\label{lem:low wtd Born}
	Under the assumptions of Proposition~\ref{prop:low wtd}, for any fixed $k \in \mathbb{N}$ we have
	\begin{align*}
		\left|\int_\R e^{-it\lambda} \chi(\lambda)\!\left(\mc{R}_0^+(\lambda) (V \mc{R}_0^+(\lambda) )^k - \mc{R}_0^-(\lambda) (V\mc{R}_0^-(\lambda) )^k\right)d\lambda\right| 
		\lesssim \frac{\la x\ra^\gamma \la y\ra^\gamma}{|t|^{1+\gamma}}.    
	\end{align*}
\end{lemma}

\begin{proof}
	
	We apply Lemma~\ref{lem:osc int} and  the Lipschitz bounds on the resolvents using $\lambda_2 = \lambda$ and $\lambda_1 = \lambda-\pi/t$ for large $|t|$, so that $|\lambda_2-\lambda_1| = \pi/|t|$ is small.  We first consider the case when $k=1$.  By \eqref{eqn:alg identity}, as in the proof of Lemma~\ref{lem:born series}, it suffices to control
	\begin{align}\label{eqn:born k1 lips}
		\int_{\R}  e^{-it\lambda} \mu_0 (\lambda) V \mc{R}_0^+(\lambda) \, d\lambda=\frac{1}{it} \int_{\R} e^{-it\lambda} \big(\partial_\lambda [\mu_0  V \mc{R}_0^+](\lambda)-\partial_\lambda [\mu_0  V \mc{R}_0^+](\lambda-\pi/t)\big) \, d\lambda.
	\end{align}
	We first consider the contribution of $\partial_\lambda \mu_0V\mR_0^+$.  To utilize the Lipschitz bounds we apply \eqref{eqn:alg identityl1l2} to see
	\begin{align}\label{eqn:born Lips example}
		[\partial_\lambda \mu_0(\lambda)-\partial_\lambda \mu_0(\lambda-\pi/t)]V\mR_0^+(\lambda)+\partial_\lambda \mu_0(\lambda-\pi/t)V[\mR_0^+(\lambda)-\mR_0^+(\lambda-\pi/t)].
	\end{align}
	Applying \eqref{eqn:mu0 bounds}, \eqref{eqf:mu Lipschitz2}, \eqref{eqn:R0 bounds}, Corollary~\ref{cor:Lips R0}, and including the spatial variable dependence, we see that
	\begin{align*}
		|\eqref{eqn:born Lips example}| \les |t|^{-\gamma}|x-z|^\gamma |\lambda| |V(z)|\bigg(\frac{1}{|z-y|^2}+\frac{|\lambda|}{|z-y|} + \frac{|\lambda|}{|z-y|^{1-\gamma}} \bigg)+|t|^{-1}|V(z)| \frac{|\lambda|^2}{|z-y| } .
	\end{align*}
	Since $|t|>1$, $|\lambda|\les 1$, and $|x-z|^\gamma\les \la x\ra^\gamma \la z\ra^\gamma$, 
	$$
	|\eqref{eqn:born Lips example}| \les |t|^{-\gamma} \la x\ra^\gamma \la z\ra^{\gamma-\delta}\bigg(\frac{1}{|z-y|^2}+\frac{1}{|z-y|^{1-\gamma}}\bigg).
	$$
	Applying Lemma~\ref{lem:spatial estimates} to control the spatial integrals, along with the support of $\chi$ shows that the contribution of \eqref{eqn:born Lips example} to \eqref{eqn:born k1 lips} is bounded by $|t|^{-1-\gamma} \la x\ra^\gamma$.  In the case that the $\lambda$ derivative acts on the resolvent on the right, applying \eqref{eqn:mu0 bounds} and Corollary~\ref{cor:Lips R0} shows that its contribution to \eqref{eqn:born k1 lips} may be bounded by
	$$
		\frac{1}{|t|^{1+\gamma}}\int_{\R^3} \la 	z\ra^{\gamma-\delta} \bigg(\frac{1}{|z-y|}+\la y\ra^\gamma \bigg)\, dz \les \frac{\la y\ra^\gamma}{|t|^{1+\gamma}}.
	$$
	A similar analysis shows that 
	$$
		\big|\mu_0(\lambda_1)V[\partial_\lambda \mR_0(\lambda_2)-\partial_\lambda\mR_0(\lambda_1)]\big| \les |t|^{-\gamma} \la y\ra^\gamma \la z\ra^{\gamma-\delta}\bigg(\frac{1}{|z-x|^2}+\frac{1}{|z-x|^{1-\gamma}}\bigg).
	$$
	Applying Lemma~\ref{lem:spatial estimates} completes the proof provided $\delta>3+2\gamma$.
	
	For $k>1$, similar to the proof of Lemma~\ref{lem:born series}, we may iterate this argument and apply Lemma~\ref{lem:spatial estimates} to see that the iterated spatial integrals are effectively harmless. Using \eqref{eqn:alg identity}, we see that we may use the Lipschitz bounds on one resolvent in the product, while the remaining resolvents are all bounded as before.  Applying the bounds \eqref{eqn:mu0 bounds}, \eqref{eqf:mu Lipschitz}, \eqref{eqn:R0 bounds} and Corollary~\ref{cor:Lips R0} repeatedly suffices to prove the claim. 
\end{proof}
We now turn to the tail of the Born series.  We do not rely on the difference between the `$+$' and `$-$' resolvents here, since we showed the iterated resolvents are locally $L^2$ in the proof of Lemma~\ref{prop:perturbed bound}, we bound both the `$+$' and `$-$' resolvent contributions in one step.  As before, the assumptions on $\Gamma(\lambda)$ are less stringent than needed in the case when zero is regular.
\begin{lemma}\label{lem:low wtd tail}
	
	Fix $0<\gamma\leq 1$, and suppose that $\Gamma(\lambda)$ and $\partial_\lambda \Gamma(\lambda)$ are absolutely bounded operators satisfying
    \begin{align*}
        \big\| | \Gamma(\lambda)| \big\|_{L^2 \to L^2} + \big\|  |\lambda \partial_\lambda \Gamma(\lambda)| \big\|_{L^2 \to L^2}\les |\lambda|^{0+},
    \end{align*}
	and the Lipschitz bounds
    \begin{align*}
        \big\||\Gamma(\lambda_2)-\Gamma(\lambda_1)|\big\|_{L^2\to L^2}+\big\| |\partial_\lambda\Gamma(\lambda_2)-\partial_\lambda\Gamma(\lambda_1)|\big\|_{L^2\to L^2} \les |\lambda_2-\lambda_1|^\gamma |\lambda_1|^{-1+},
    \end{align*}
	when $0<|\lambda_1|\leq |\lambda_2|\leq 1$.  If $|V(x)| \les \la x\ra^{-\delta}$ for some $\delta > 3+2\gamma$, then for $|t|>1$, we have the bound
	\begin{align*}
		\bigg| \int_{\R} e^{-it\lambda} \chi(\lambda) \mR_0^\pm(\lambda) V \mR_0^\pm(\lambda) v^* \Gamma(\lambda)  v \mR^{\pm}_0 (\lambda) V  \mR_0^\pm(\lambda)\, d\lambda \bigg| \les \frac{\la x\ra^\gamma \la y\ra^\gamma}{|t|^{1+\gamma}}.
	\end{align*}
	
\end{lemma}

\begin{proof}
	
	Again, we reduce this to an application of Lemma~\ref{lem:osc int} and the Lipschitz bounds for the resolvents and $\Gamma(\lambda)$.  We consider only the `+' case, the `$-$' follows identically. The first assumptions on $\Gamma$ and \eqref{eqn:R0 bounds} ensure there are no boundary terms at zero when integrating by parts. After one application of integration by parts, we have two cases to consider.  Either the derivative acts on a resolvent, or it acts on $\Gamma(\lambda)$.  Consider the first case, here we note that it suffices to consider 
	\begin{align}\label{eqn:Gamma Lips}
		\bigg| \int_{\R} e^{-it\lambda} \chi(\lambda) \partial_\lambda [\mR_0^+(\lambda) V \mR_0^+(\lambda)] v^* \Gamma(\lambda)  v \mR^{+}_0 (\lambda) V  \mR_0^+(\lambda)\, d\lambda \bigg|.
	\end{align}
	Noting the Lipschitz bounds in Lemma~\ref{lem:Minv zeroreg}, it suffices to show that the iterated resolvent and its derivative satisfy appropriate Lipschitz bounds.  Applying \eqref{eqn:alg identity}, we see that
	\begin{multline}\label{eqn:iterated R0 Lip}
		(\partial_\lambda^j \mR_0^+)V\mR_0^+(\lambda_2)-(\partial_\lambda^j \mR_0^+)V\mR_0^+(\lambda_1)\\
		=[\partial_\lambda^j\mR_0^+(\lambda_2)-\partial_\lambda^j\mR_0^+(\lambda_1)]V\mR_0^+(\lambda_2)+\partial_\lambda^j\mR_0^+(\lambda_1)V[\mR_0^+(\lambda_2)-\mR_0^+(\lambda_1)].
	\end{multline}
	When $j=0$ applying the bounds in \eqref{eqn:R0 bounds}, Corollary~\ref{cor:Lips R0} and Lemma~\ref{lem:spatial estimates}, to $\lambda_2 = \lambda$ and $\lambda_1 = \lambda - \pi/t$ when $|t|>1$ we see that
	\begin{align*}
		|\eqref{eqn:iterated R0 Lip}|    \les  |t|^{-\gamma}  \int_{\R^3}\frac{1+|x-z_1|^\gamma  }{|x-z_1|}\la z_1\ra^{-\delta}\frac{1+|x-y|}{|x-y|^2}\, dz_1 \les |t|^{-\gamma}  |x-y|^{-1}.
	\end{align*}
	To apply Lemma~\ref{lem:spatial estimates} when $k+\ell=3$, we used the crude bound $a^{-1}b^{-2}\les a^{-2}b^{-2}+a^{-1}b^{-1}$ for $a,b>0$.
	Applying a similar argument when $j=1$ results in a weight in $x$, namely the first summand contributes
	\begin{align*}
		|[\partial_\lambda\mR_0^+(\lambda_2)-\partial_\lambda\mR_0^+(\lambda_1)]V\mR_0^+(\lambda_2)| \les  |t|^{-\gamma}\!\!  \int_{\R^3}  (1+|x-z_1|^\gamma)  \la z_1\ra^{-\delta}\frac{1+|z_1-y|}{|z_1-y|^2}\, dz_1 \les |t|^{-\gamma}   \la x\ra^\gamma.
	\end{align*}
	From this, through an application of Lemma~\ref{lem:GolVis}, we can see the Lipschitz bounds of the $L^2$ norms of iterated resolvents
	\begin{align}
		\| v(\,\cdot\,) \mR_0^+V\mR_0^+(\lambda_2)-v(\,\cdot\,) \mR_0^+V\mR_0^+(\lambda_1)  \|_{L^2} &\les |t|^{-\gamma},\label{eqn:L2 iter Lips1}\\
		\| v(\,\cdot\,)\big(\partial_\lambda [ \mR_0^+V\mR_0^+(\lambda_2)](\,\cdot\,,y)- \partial_\lambda [\mR_0^+V\mR_0^+(\lambda_1)](\,\cdot\,,y)\big)  \|_{L^2} &\les |t|^{-\gamma} \la y\ra^{\gamma}.\label{eqn:L2 iter Lips2}
	\end{align}
	Noting that the spatial integrals that arise in applying Lemma~\ref{lem:osc int} and \eqref{eqn:alg identity} to \eqref{eqn:Gamma Lips} may be controlled by a sum of terms of the form:
    \begin{align*}
        &\big\| \big(\partial_\lambda^{j_1} [ \mR_0^+V\mR_0^+(\lambda_2)](x,\,\cdot\,)-\partial_\lambda^{j_1} [\mR_0^+V\mR_0^+(\lambda_1)](\,\cdot\,,x) \big) v^*(\,\cdot\,) \big\|_{L^2}\\
		&\qquad \qquad\qquad\qquad \times
		\big\| | \partial_\lambda^{j_2}\Gamma(\lambda) | \big\|_{L^2\to L^2}
		\big\| v(\,\cdot\,)\partial_\lambda^{j_3}[ \mR_0^+V\mR_0^+(\lambda) ] \big\|_{L^2}
		\\
		& + \big\| \partial_\lambda^{j_1} [ \mR_0^+V\mR_0^+(\lambda)](x,\,\cdot\,) v^*(\,\cdot\,) \big\|_2
		\big\| | \partial_\lambda^{j_2}\big(\Gamma(\lambda_2)-\Gamma(\lambda_1) \big) | \big\|_{L^2\to L^2}
		\big\| v(\,\cdot\,)\partial_\lambda^{j_3}[ \mR_0^+V\mR_0^+(\lambda) ] \big\|_{L^2}
		\\
		&  + \big\| \partial_\lambda^{j_1} [ \mR_0^+V\mR_0^+(\lambda)](x,\,\cdot\,) v^*(\,\cdot\,) \big\|_{L^2}
		\big\| | \partial_\lambda^{j_2}\Gamma(\lambda) |\big\|_{L^2\to L^2}
		\big\| v(\,\cdot\,)\partial_\lambda^{j_3}[ \mR_0^+V\mR_0^+(\lambda_2)-\mR_0^+V\mR_0^+(\lambda_1)] \|_{L^2},
    \end{align*}
	where $j_1,j_2,j_3\in\{0,1\}$ and $j_1+j_2+j_3=1$.  Combining equations \eqref{eqn:L2 iter Lips1}, \eqref{eqn:L2 iter Lips2}, the assumptions on $\Gamma$, and the support of $\chi$ show that
	\begin{align*}
		|\eqref{eqn:Gamma Lips}| \lesssim \bigg| \int_{\R} e^{-it\lambda} \mathcal E(\lambda)\, d\lambda \bigg|,
	\end{align*}
	where $\mathcal E(\lambda)$ is supported on $(-1,1)$ and $\mathcal E'(\lambda)$ is integrable with
	\begin{align*}
		|\mathcal E'(\lambda)-\mathcal E'(\lambda-\pi/t)| \les |t|^{-\gamma}\la x\ra^\gamma \la y\ra^\gamma |\lambda|^{-1+}.
	\end{align*}
	Applying Lemma~\ref{lem:osc int} proves the claim.
\end{proof}
We note here that the Lipschitz bounds used for the resolvents in the proof, Lemma~\ref{lem:R0 error Lip} and Corollary~\ref{cor:Lips R0}, the extra smallness in $\lambda$ was not used.  That is, we dominate all positive powers of $|\lambda_2|$ by a constant in this proof.  
We are now ready to prove Proposition~\ref{prop:low wtd}.
\begin{proof}[Proof of Proposition~\ref{prop:low wtd}]
	By expanding $\mR_V^\pm$ into a Born series expansion as in \eqref{eqn:born series}, we can control the contribution of each term.  The contribution of first term in \eqref{eqn:born series} to \eqref{eqn:stone} is controlled by Theorem~\ref{thm:free}, the contribution of the second and third are controlled by Lemma~\ref{lem:low wtd Born}.  For the final term, we do not utilize the difference between the `+' and `$-$' resolvent, but control each by applying Lemma~\ref{lem:low wtd tail}.
\end{proof}

We note that one can apply the Lipschitz bounds directly without integrating by parts to prove weaker versions of this theorem that require less decay on the potential as in Theorem~\ref{thm:weak results}.  Namely,

\begin{theorem}\label{thm:weak decay}
	
	Fix a value of $0\leq \gamma \leq 1$. If zero is regular and $|V(x)|\les \la x\ra^{-\delta}$ for some $\delta>1+2\gamma$, then
	\begin{align*}
		\|e^{-it\mathcal H}P_{ac}(\mathcal H)\chi(\mathcal H)\|_{L^1 \to L^\infty} \les \la t\ra^{-\gamma}.
	\end{align*}
\end{theorem}
This shows, for weaker decay on the potential, that the low energy portion of the evolution may be controlled.  In particular, the evolution is bounded if $\delta>1$.

\begin{proof}
	
	Instead of applying Lemma~\ref{lem:osc int} as in the proofs of the previous theorems, if $\mathcal E(\lambda)$ is bounded we instead apply
    \begin{align*}
        \bigg|\int_{\R}e^{-it\lambda} \mathcal E(\lambda)\, d\lambda \bigg| \les \int_{\R} \bigg| \mathcal E(\lambda)-\mathcal E\bigg(\lambda - \frac{\pi}{t}\bigg)\bigg| \, d\lambda
    \end{align*}
	to the Stone's formula, \eqref{eq:Stone}.  This follows from the proof of Lemma~\ref{lem:osc int} without integrating by parts first.
	
	Here, instead of iterating the resolvent identity directly, we note that we may write
	\begin{align}\label{eqn:RL RH defn}
		\mR_0^\pm(\lambda)(x,y)&=\chi(\lambda|x-y|)\mR_0^\pm(\lambda)(x,y)+\widetilde \chi(\lambda|x-y|)\mR_0^\pm(\lambda)(x,y)\\
		&\!:=\mR_L^\pm(\lambda)(x,y)+\mR_H^\pm(\lambda)(x,y).\nn
	\end{align}
	Here $\widetilde \chi = 1-\chi$ is a smooth cut-off away from a neighborhood of zero.
	By the expansions developed in Lemma~\ref{lem:R0 expansions}, we have (for $k=0,1$)
	\begin{align}\label{eqn:RLRH bds}
		|\partial_\lambda^k \mR_L^\pm(\lambda)(x,y)|\les |x-y|^{k-2}, \qquad 	|\partial_\lambda^k \mR_H^\pm(\lambda)(x,y)|\les |\lambda|\, |x-y|^{k-1}.
	\end{align}
	In particular, we note that $\mR_H^\pm$ is a locally $L^2$ function of $x$ or $y$.  As before, we may use these bounds to obtain Lipschitz bounds (for $|\lambda_1|\leq |\lambda_2|\leq 1$ and any $0\leq \gamma\leq 1$)
	\begin{align}\label{eqn:RL lips}
		|\mR_L^\pm(\lambda_2)(x,y)-\mR_L^\pm(\lambda_1)(x,y)|&\les \frac{|\lambda_2-\lambda_1|^\gamma}{|x-y|^{2-\gamma}},\\
		|\mR_H^\pm(\lambda_2)(x,y)-\mR_H^\pm(\lambda_1)(x,y)|&\les \frac{|\lambda_2|\,|\lambda_2-\lambda_1|^\gamma}{|x-y|^{1-\gamma}}. \label{eqn:RH lips}
	\end{align}
	From here we may selectively iterate the symmetry resolvent identity: 
	$$
	\mR_V^\pm(\lambda)V=\mR_0^\pm(\lambda)v^*(M^\pm)^{-1}(\lambda )v
	$$
	to form a Born series expansion tailored to optimize the decay needed from the potential.
	\begin{multline}\label{eqn:Born tailored}
		\mR_V^\pm(\lambda) = \mR_0^\pm(\lambda)-\mR_H^\pm (\lambda)v^* (M^\pm)^{-1}(\lambda)v \mR_H^\pm(\lambda)
		-\mR_L^\pm (\lambda)V\mR_0 v^* (M^\pm)^{-1}(\lambda)v \mR_H^\pm(\lambda)\\
		-\mR_H^\pm (\lambda)v^* (M^\pm)^{-1}(\lambda)v \mR_0^\pm V \mR_L^\pm(\lambda)
		+\mR_L^\pm (\lambda)V\mR_0^\pm(\lambda)v^* (M^\pm)^{-1}(\lambda)v \mR_0^\pm(\lambda)V \mR_L^\pm(\lambda).
	\end{multline}
	Using \eqref{eqn:RL RH defn} and Lemma~\ref{lem:GolVis} $\mR_H^\pm(\lambda)(x,\,\cdot\,)v^*(\,\cdot\,)$ is in $L^2$ uniformly in $x$ provided $\delta > 1/2$.
	On the other hand, using \eqref{eqn:RH lips} and Lemma~\ref{lem:GolVis} shows that
	$$
	\|\mR_H^\pm(\lambda_2)(x,\,\cdot\,)v^*(\,\cdot\,)-\mR_H^\pm(\lambda_1)(x,\,\cdot\,)v^*(\,\cdot\,)\|_2 \les |\lambda_2-\lambda_1|^\gamma |\lambda_2|
	$$
	uniformly in $x$ provided $\delta > \gamma + 1/2$.
	Applying \eqref{eqn:alg identityl1l2} we see that
	\begin{multline*}
		\mR_L^\pm (\lambda_2)V\mR_0(\lambda_2) v^*-\mR_L^\pm (\lambda_1)V\mR_0(\lambda_1) v^*\\
		=[\mR_L^\pm (\lambda_2)-\mR_L^\pm (\lambda_1)]V \mR_0^\pm (\lambda_2)v^*+\mR_L^\pm (\lambda_1)V[\mR_0(\lambda_2)-\mR_0(\lambda_2)] v^* 
	\end{multline*}
	By \eqref{eqn:R0 bounds}, \eqref{eqn:RLRH bds}, \eqref{eqn:RL lips}, and Corollary~\ref{cor:Lips R0} we see that the size of the integral kernel of this operator is bounded by
	\begin{align*}
		|\lambda_2-\lambda_1|^\gamma |v^*(z_2)|\,  \int_{\R^3} |V(z_1)| \, \bigg( \frac{1+|z_1-z_2|}{|x-z_1|^{2-\gamma}|z_1-z_2|^2}+ \frac{1+|z_1-z_2|^\gamma}{|x-z_1|^2 |z_1-z_2| } \bigg)\, dz_1.
	\end{align*}
	Applying Lemma~\ref{lem:spatial estimates}, we note that the integration smooths out the local singularity enough to be locally $L^2$.  The decay of the resulting upper bound in terms of $x, z_2$ is constrained by the case when $k=2$ and $\ell=1-\gamma$.  From this we see that
	if $\delta>1+\gamma$ we have the upper bound
	$$
	\big|[\mR_L^\pm (\lambda_2)V\mR_0(\lambda_2) v^*-\mR_L^\pm (\lambda_1)V\mR_0(\lambda_1) v^*](x,z_2)\big|\les |\lambda_2-\lambda_1|^\gamma \frac{\la z_2\ra^{-\frac{1+\gamma}{2}-}}{|x-z_2|^{1-\gamma}}.
	$$
	Applying Lemma~\ref{lem:GolVis} we see
	$$
	\sup_{x\in \R^3}\|[\mR_L^\pm (\lambda_2)V\mR_0(\lambda_2)(x,\,\cdot\,) -\mR_L^\pm (\lambda_1)V\mR_0(\lambda_1)(x,\,\cdot\,)]v^*(\,\cdot\,) \|_{L^2}\les |\lambda_2-\lambda_1|^\gamma
	$$
	A similar analysis shows that if $\delta>1$, then
	$$
	\sup_{x\in \R^3}\|\mR_L^\pm (\lambda)V\mR_0(\lambda)(x,\,\cdot\,) \|_{L^2}\les 1.
	$$
	We further require $\delta>1+2\gamma$ to obtain the Lipschitz bounds on $(M^\pm )^{-1}(\lambda)$ in Lemma~\ref{lem:Minv zeroreg}.
	
	The claim now follows by selecting $\lambda_2=\lambda$ and $\lambda_1=\lambda-\pi/t$ for $|t|>1$, here
	$\mathcal E(\lambda)=\chi(\lambda)[\mR_V^+-\mR_V^-](\lambda)$.  The spatial integrals are controlled by the $L^2$ norms found above with an analysis similar to that in the proof of Lemma~\ref{prop:perturbed bound}.
\end{proof}

\section{Dispersive bounds when zero is not regular}\label{sec:eval}

We now consider the low energy evolution when zero energy is not regular.  As shown in Section~\ref{sec:thresholds} below, if zero energy is not regular the operator $\mathcal H$ has a zero energy eigenvalue.  This further complicates the inversion process near the threshold, and results in an expansion for the spectral measure that is singular as $\lambda\to 0$.  We show that this singularity may be overcome, with only a slight increase in the needed decay on $V$.  We show this by utilizing the cancellation between the `$+$' and `$-$' resolvents that was not needed in the regular case to overcome the loss of powers of $\lambda$ that arise due to the presence of a zero energy eigenvalue.

The main result of this section are the dispersive bounds when zero is not regular.
\begin{prop}\label{prop:zero not reg dispersive bounds}	
	Assume that $|V(x)|\les \la x\ra^{-\delta}$ for some $\delta>3$. If zero is not a regular point of $\mathcal H$, then
	\begin{align*}
		\sup_{x,y\in \R^3} \left| \int_{\R} e^{-it\lambda} \chi(\lambda) (\mR_V^+-\mR_V^-)(\lambda)(x,y)\, d\lambda\right| \les \la t\ra^{-1}.
	\end{align*}
	Further, for fixed $0\leq \gamma< \frac12$, if $\delta>3+4\gamma$, then for  $|t|>1$ we have the weighted bound
	\begin{align*}
		\sup_{x,y\in \R^3} \left| \int_{\R} e^{-it\lambda} \chi(\lambda) (\mR_V^+-\mR_V^-)(\lambda)(x,y)\, d\lambda\right| \les \frac{\la x\ra^\gamma \la y\ra^\gamma}{\la t\ra^{1+\gamma}}.
	\end{align*}
\end{prop}

These dispersive bounds follow by developing an appropriate expansion for the operators $(M^\pm(\lambda))^{-1}$ that account for the existence of the zero energy eigenvalues in Proposition~\ref{prop:Minv eval} below.

To invert $M^\pm(\lambda)=U+v\mR_0^\pm(\lambda^2)v^*$, for small $\lambda$, we use the
following  Lemma (see Lemma 2.1  in \cite{JN}) with $S=S_1$, the Riesz projection onto the kernel of $T_0=M^\pm(0)=U+v\mG_0 v^*$.

\begin{lemma}\label{JNlemma}
	Let $M$ be a closed operator on a Hilbert space $\mathcal{H}$ and $S$ a projection. Suppose $M + S$ has a bounded
	inverse. Then $M$ has a bounded inverse if and only if
	$$
	   B \coloneqq S - S(M + S)^{-1}S
	$$
	has a bounded inverse in $S\mathcal{H}$, and in this case
	$$
	   M^{-1} = (M + S)^{-1} + (M + S)^{-1}SB^{-1}S(M + S)^{-1}.
	$$
\end{lemma}

Here, we have
$$
	B^\pm(\lambda)=S_1-S_1(M^\pm(\lambda)+S_1)^{-1}S_1.
$$
We note that, with a slight abuse of notation,  the expansions for $(M^\pm)^{-1}(\lambda)$ in Lemma~\ref{lem:Minv zeroreg} all hold for $(M^\pm(\lambda)+S_1)^{-1}$ with $D_1=(T_0+S_1)^{-1}$.  When zero is regular, $S_1=0$, so the definitions agree in this case.

For the Lipschitz bounds we note the following fact about products of Lipschitz functions.  
\begin{lemma}\label{lem:lipschitz lem}
	
	If $f,g$ are functions supported on $0<|\lambda|\ll 1$ with $|f(\lambda_2)-f(\lambda_1)|\leq C_f |\lambda_2-\lambda_1|^\gamma$, $|g(\lambda_2)-g(\lambda_1)|\leq C_g |\lambda_2-\lambda_1|^\gamma$.  If, for all $0<|\lambda|\ll 1$, we have $|f(\lambda)|\leq M_f$ and $|g(\lambda)|\leq M_g$, then
	$$
		|fg(\lambda_2)-fg(\lambda_1)|\les (C_gM_f +C_fM_g )|\lambda_2-\lambda_1|^\gamma.
	$$
	
\end{lemma}
The proof follows from \eqref{eqn:alg identityl1l2} and the triangle inequality:
$$
	|fg(\lambda_2)-fg(\lambda_1)|=|f(\lambda_1)[g(\lambda_2)-g(\lambda_1)]+[f(\lambda_2)-f(\lambda_1)]g(\lambda_2)|.
$$
The quantities $M_f,M_g$ may be functions of $\lambda$ that become singular as $\lambda \to 0$.  The same argument may be applied to see that
$$
	|\partial_\lambda (fg)(\lambda_2)-\partial_\lambda (fg)(\lambda_1)| \les (C_gM_{f'}+C_{f'}M_g+C_fM_{g'}+C_{g'}M_f)|\lambda_2-\lambda_1|^\gamma.
$$
When we use this in the expansions below, only one of these bounds gets large near zero, the remaining quantities are bounded.  This allows us to push forward Lipschitz bounds while also accounting for any singularities.

To obtain the dispersive bounds, we need a slightly longer expansion for $(M\pm(\lambda)+S_1)^{-1}$.
\begin{lemma}\label{M+S1 inverse}
	
	Assume that zero is not regular, and that $|V(x)|\les \la x\ra^{-\delta}$.  For fixed $0\leq \ell \leq 1$, for sufficiently small $|\lambda|$ if $\delta>3+2\ell$ we have
    \begin{align*}
        (M^\pm(\lambda)+S_1)^{-1}
        & = D_1 + \lambda D_1v\mG_1v^*D_1+\lambda^2 D_1 \Gamma_2^\pm D_1 + D_1M_{1,\ell}^\pm (\lambda)D_1,
    \end{align*}
	where $\Gamma_2^\pm$ are $\lambda$ independent, absolutely bounded operators, and $M_{1,\ell}^\pm(\lambda)$ satisfies
	$$
		\| \partial_\lambda^k M_{1,\ell}^\pm(\lambda)\|_{HS}\les \lambda^{2+\ell-k}, \qquad k=0,1.
	$$
	Furthermore, if $\delta>3+2(\gamma+\ell(1-\gamma))$ for some $0\leq \gamma \leq 1$, then for   $0<|\lambda_1|\leq |\lambda_2|\ll 1$, we have
	\begin{align*}
		\|  M_{1,\ell}^\pm(\lambda_2)-M_{1,\ell}^\pm(\lambda_1)\|_{HS}&\les  |\lambda_2-\lambda_1|^\gamma |\lambda_2|^{1+\gamma+\ell},\\
		\|  \partial_\lambda M_{1,\ell}^\pm(\lambda_2)-\p_\lambda M_{1,\ell}^\pm(\lambda_1)\|_{HS}&\les  |\lambda_2-\lambda_1|^\gamma |\lambda_2|^{(1+\ell)(1-\gamma)}.
	\end{align*}	
\end{lemma}	
Here we keep the first two terms in the expansion explicit since their exact form will be important when we  invert the operators $B^\pm(\lambda)$.

    \begin{proof}
    Recall the definition $M^\pm(\lambda)$ in \eqref{eq:pert1}, the second expansion for $\mR_0^\pm$ in Lemma~\ref{lem:R0 expansions}, and that $T_0+S_1$ is invertible on $L^2$ from Definition~\ref{def:resonances}.      
    With a slight abuse of notation, we write $D_1=(T_0+S_1)^{-1}$.  This agrees with our previous notation in the case when $S_1=0$.  Expanding in a Neumann series, we have
        \begin{align*}
    	    (M^\pm(\lambda)+S_1)^{-1}
            % & = (T_0 + M_0^\pm(\lambda))^{-1}\\
            & = \big(T_0 + S_1 + \lambda v \mG_1 v^* + i \lambda^2 v \mG_2^\pm v^* + v\mE_2^\pm(\lambda)v^*\big)^{-1}\\
            & = \big(\mathbbm 1 + D_1(\lambda v \mG_1 v^* + i \lambda^2 v \mG_2^\pm v^* + v\mE_2^\pm(\lambda)v^*\big)^{-1}D_1\\
            &=D_1-\lambda D_1v\mG_1 v^* D_1+\lambda^2D_1[v\mG_1 v^*D_1 v\mG_1 v^*-iv\mG_2^\pm v^*]D_1+D_1M_{1,\ell}^\pm(\lambda)D_1
    	\end{align*}
    	Here one needs $\delta>3+2\ell$ to ensure that the operator with integral kernel $v(x)\mE_2^\pm(\lambda,|x-y|)v^*(y)$ is Hilbert-Schmidt.  The bounds on the error term follow from the error bounds on $\mE_2^\pm(\lambda)$ in Lemma~\ref{lem:R0 expansions}, the subscript $\ell$ indicates the extra powers of $\lambda$ in its upper bounds. 
    	
    	The Lipschitz bounds follow from Lemma~\ref{lem:lipschitz lem} since we may write $M_{1,\ell}^\pm(\lambda)$ as a combination of absolutely bounded operators with powers of $\lambda$ and $v\mE_2^\pm(\lambda)v^*$.  In these combinations, any term with no $v\mE_2^\pm(\lambda)v^*$ has at least three powers of $\lambda$.  The Lipschitz bounds on $\mE_2^\pm(\lambda)$ and its derivative in Lemma~\ref{lem:R0 error Lip}, along with the fact that Lipschitz bounds apply to functions of the form $f(\lambda)=\lambda^k$ for any $k\in \mathbb N$.  All terms in $M_{1,\ell}$ are dominated by the contribution of $v\mE_2^\pm(\lambda)v^*$.
    \end{proof}

\begin{lemma}\label{lem: B inverse}
	
	Assume that zero is not regular, and that $|V(x)|\les \la x\ra^{-\delta}$.  For fixed $0\leq \ell \leq 1$, for sufficiently small $0<|\lambda|\ll1$ if $\delta>3+2\ell$ we have
	\begin{align*}
		(B^\pm(\lambda) )^{-1}
		& =- \frac{D_2}{\lambda}+D_2\Gamma_0^\pm D_2+B_{-1,\ell}^\pm(\lambda),
	\end{align*}
where $\Gamma_0^\pm$ are $\lambda$ independent absolutely bounded operators.
$$
\| \partial_\lambda^k B_{-1,\ell}^\pm(\lambda)\|_{HS}\les \lambda^{\ell-k}, \qquad k=0,1.
$$
Furthermore for fixed $0\leq \gamma \leq 1$ if $\delta>3+2(\gamma+\ell(1-\gamma))$, then for   $0<|\lambda_1|\leq |\lambda_2|\ll 1$, we have
\begin{align*}
	\|  B_{-1,\ell}^\pm(\lambda_2)-B_{-1,\ell}^\pm(\lambda_1)\|_{HS}&\les  |\lambda_2-\lambda_1|^\gamma |\lambda_1|^{ (1-\gamma)(\ell -1)}\\
	\|  \partial_\lambda B_{-1,\ell}^\pm(\lambda_2)-\p_\lambda B_{-1,\ell}^\pm(\lambda_1)\|_{HS}&\les  |\lambda_2-\lambda_1|^\gamma |\lambda_1|^{\ell(1-\gamma)-\gamma-1}.
\end{align*}
\end{lemma}	
\begin{proof}
Recalling that $S_1D_1 = D_1S_1 = S_1$, to use the inversion technique of Lemma~\ref{JNlemma}, we first need to invert the operators
\begin{align*}
    B^\pm(\lambda)
    & = S_1-S_1(M^\pm(\lambda)+S_1)^{-1}S_1\\
    & = S_1-S_1 \bigg[D_1 + \lambda D_1v\mG_1v^*D_1+\lambda^2 D_1 \Gamma_2^\pm D_1 + D_1M_{1,\ell}^\pm (\lambda)D_1\bigg]S_1\\
    & = -\lambda S_1 v \mathcal{G}_1 v^* S_1 - \lambda^2 S_1 \Gamma_2^\pm S_1 + B^\pm_{1,\ell}(\lambda)
\end{align*}
The leading $S_1$ is canceled out by the leading contribution of the second term.
Here $B_{1,\ell}^\pm(\lambda)$ obeys the same bounds as $M_{1,\ell}^\pm(\lambda)$ since $S_1$ is a $\lambda$ independent $L^2$-bounded projection.  Defining $B_{0,\ell}^\pm(\lambda)=-\lambda^{-1}B_{1,\ell}^\pm(\lambda)$, we see that
$$
\| \partial_\lambda^k B_{0,\ell}^\pm(\lambda)\|_{HS}\les \lambda^{1+\ell-k}, \qquad k=0,1.
$$
By Lemma~\ref{lem:invertibility} below, the operator $S_1v\mG_1v^*S_1$ is invertible on $S_1L^2$.  We denote $D_2:=(S_1v\mG_1v^*S_1)^{-1}$.  Then, by a Neumann series expansion we have
\begin{align*}
    (B^\pm(\lambda))^{-1}
     & = -\frac{1}{\lambda}\left[S_1 v \mathcal{G}_1 v^* S_1 + \lambda S_1 \Gamma_2^\pm S_1 + B^\pm_{0,\ell}(\lambda)\right]^{-1} \\
     & = -\frac{1}{\lambda}\left[\mathbbm 1+\lambda D_2 S_1 \Gamma_2^\pm S_1 + D_2B^\pm_{0,\ell}(\lambda)\right]^{-1} D_2 
     =- \frac{D_2}{\lambda}+D_2\Gamma_0^\pm D_2+B_{-1,\ell}^\pm(\lambda),
\end{align*}
where we collect all the error terms from the Neumann series into $B_{-1,\ell}^\pm$, which by the error bounds in Lemma~\ref{M+S1 inverse} yields
$$
	\| \partial_\lambda^k B_{-1,\ell}^\pm(\lambda)\|_{HS}\les \lambda^{\ell-k}, \qquad k=0,1.
$$
For the Lipschitz bounds, we need to consider cases since the power on $\lambda$ may be negative.  The error term here is composed of products of $\lambda^{-2} M_{1,\ell}^\pm(\lambda)$ and operators of the form $\lambda \Gamma_0^\pm$ with $\Gamma_0^\pm$ independent of $\lambda$.  As such, the limiting factor is the Lipschitz behavior of $\lambda^{-2}v\mE^\pm_2(\lambda) v^*$, since the remaining terms are dominated by this one.  We write $A(\lambda)= \lambda^{-2}\mE_2^\pm(\lambda)(x,y)$ and $r=|x-y|$ to illustrate where the more stringent decay conditions on $V$ arise.

We consider cases.  First, if $|\lambda_2-\lambda_1|\approx |\lambda_2|$, then either  $|\lambda_1|\ll |\lambda_2|$ or $\lambda_1$ and $\lambda_2$ have opposite signs with $|\lambda_1|\approx |\lambda_2|$.  In either case, by Lemma~\ref{lem:R0 expansions} we have
$$
	|A(\lambda_2)-A (\lambda_1)|\les |\lambda_2|^{\ell} r^{\ell}\approx |\lambda_2-\lambda_1|^\gamma |\lambda_j|^{\ell-\gamma}r^\ell,
$$
where $\lambda_j=\lambda_2$ if the exponent is positive and $\lambda_1$ if the exponent is negative.  On the other hand, if $|\lambda_2-\lambda_1|\ll |\lambda_2|$ then we must have $|\lambda_1|\approx |\lambda_2|$ where $\lambda_1$ and $\lambda_2$ have the same sign.  We may then use the mean value theorem to write
$$
	|A(\lambda_2)-A(\lambda_1)|=\bigg| \int_{\lambda_1}^{\lambda_2} \partial_\lambda A(s)\, ds \bigg|\les |\lambda_2-\lambda_1||\lambda_1|^{\ell-1} r^\ell;
$$
we note that zero is not in the interval over which we integrate.  Interpolating that with the bound from the triangle inequality of $|\lambda_2|^\ell r^\ell$ yields the bound that is dominated by $|\lambda_2-\lambda_1|^\gamma |\lambda_1|^{(1-\gamma)(\ell-1)} r^\ell$.  Here we note that the singular behavior of the derivative can't be improved in the interpolation process since the upper bound from the triangle inequality does not involve powers of $\lambda_1$. 

A similar case analysis with the derivative shows the second Lipschitz bound:
$$
	|\partial_\lambda A(\lambda_2)-\partial_\lambda A(\lambda_1)|\les |\lambda_2-\lambda_1|^\gamma |\lambda_1|^{(\ell-1)(1-\gamma)-2\gamma}\la r\ra^{\gamma+\ell(1-\gamma)},
$$
where the power on $\lambda_1$ simplifies to $\ell(1-\gamma)-\gamma-1$.  The assumptions on $\delta$ are need to ensure that the kernel $v(x)\la x-y\ra^{\gamma + \ell(1-\gamma)}v^*(y)$ is Hilbert-Schmidt.  The contribution of the remaining terms in $B_{-1,\ell}^\pm(\lambda)$ are dominated by these bounds.
\end{proof}
The preceding lemmas serve to prove the following expansion of $(M^{\pm})^{-1}(\lambda)$ in the presence of a zero energy eigenvalue.
\begin{prop}\label{prop:Minv eval}
	Assume that zero is not regular, and that $|V(x)|\les \la x\ra^{-\delta}$.  For fixed $0\leq \ell \leq 1$, for sufficiently small $0<|\lambda|$ if $\delta>3+2\ell$ we have
    \begin{align*}
        (M^\pm)^{-1}(\lambda) = -\frac{D_2}{\lambda} + \widetilde{\Gamma}_{0}^\pm+M_{-1,\ell}^\pm(\lambda),
    \end{align*}
	where 
	$$
	\| \partial_\lambda^k M_{-1,\ell}^\pm(\lambda)\|_{HS}\les \lambda^{\ell-k}, \qquad k=0,1.
	$$
	Furthermore for fixed $0\leq \gamma \leq 1$ if $\delta>3+2(\gamma+\ell(1-\gamma))$, then for   $0<|\lambda_1|\leq |\lambda_2|\ll 1$, we have
	\begin{align*}
		\|  M_{-1,\ell}^\pm(\lambda_2)-M_{-1,\ell}^\pm(\lambda_1)\|_{HS} & \les |\lambda_2-\lambda_1|^\gamma |\lambda_1|^{ (1-\gamma)(\ell -1)},\\
		\|  \partial_\lambda M_{-1,\ell}^\pm(\lambda_2)-\p_\lambda M_{-1,\ell}^\pm(\lambda_1)\|_{HS} & \les |\lambda_2-\lambda_1|^\gamma |\lambda_1|^{\ell(1-\gamma)-\gamma-1}.
	\end{align*}
\end{prop}
 We note that the difference between the `$+$' and `$-$' terms is not used in the arguments in Section~\ref{sec:zero reg} when zero is regular.  Hence the $\pm$ dependence of the order $\lambda^0$ term will not affect our ability to use these bounds in this case.
\begin{proof}
Using Lemmas~\ref{JNlemma}, \ref{M+S1 inverse}, and \ref{lem: B inverse}, we have
    \begin{multline*}
        (M^\pm)^{-1}(\lambda)
         = D_1 + \lambda D_1v\mG_1v^*D_1+\lambda^2 D_1 \Gamma_2^\pm D_1 + D_1M_{1,\ell}^\pm (\lambda)D_1 \\
         + \left(D_1 + \lambda D_1v\mG_1v^*D_1+\lambda^2 D_1 \Gamma_2^\pm D_1 + D_1M_{1,\ell}^\pm (\lambda)D_1\right) 
        S_1\left(- \frac{D_2}{\lambda}+D_2\Gamma_0^\pm D_2+B_{-1,\ell}^\pm(\lambda)\right) S_1 \\
         \left(D_1 + \lambda D_1v\mG_1v^*D_1+\lambda^2 D_1 \Gamma_2^\pm D_1 + D_1M_{1,\ell}^\pm (\lambda)D_1\right)
    \end{multline*}
    Expanding this, the most singular contribution with respect to the spectral parameter is
    \begin{equation*}
        -D_1\frac{S_1D_2S_1}{\lambda}D_1=-\frac{S_1D_2S_1}{\lambda}=-\frac{D_2}{\lambda}.
    \end{equation*}
    The next largest contribution with respect to the spectral parameter is the $\lambda^0$ term,
    \begin{equation*}
        \widetilde{\Gamma}_0^{\pm} \coloneqq
        D_1-D_1S_1D_2S_1D_1v\mathcal{G}_1 v^* D_1-D_1v\mathcal{G}_1 v^*D_1S_1D_2S_1D_1+D_1S_1D_2\Gamma_0^{\pm}D_2S_1D_1.
    \end{equation*}
    The remaining terms form the error $M^{\pm}_{-1,\ell}$. The error and its first derivative are dominated by the contribution of  $D_1S_1B_{-1, \ell}^{\pm}(\lambda)S_1D_1$. The first claim on the error term follows from Lemma~\ref{lem: B inverse}.    
    By an application of Lemma~\ref{lem:lipschitz lem}, the Lipschitz bounds on $B_{-1, \ell}^{\pm}$ in Lemma~\ref{lem: B inverse} and the absolute boundedness of the various operators suffice to show the Lipschitz bounds for $M_{-1, \ell}^{\pm}$.
\end{proof}
It is convenient to define the function 
$$
	\mu_1(\lambda,x,y)=\lambda^{-1}\mu_0(\lambda,x,y)=\lambda^{-1}\chi(\lambda)[\mR_0^+-\mR_0^-](\lambda)(x,y).
$$
\begin{lemma}\label{lem:mu1 bounds}
	The following bounds hold:
	\begin{align*}
		\left|{\mu_1(\lambda, x, y)}\right| \lesssim \min\!\left(|\lambda|, \frac{1}{|x - y|}\right), 
		\qquad 
		\left|\partial_\lambda\mu_1(\lambda, x, y)\right| \lesssim 1,
		\qquad
		\left|\partial_\lambda^2 \mu_1(\lambda, x, y)\right| \lesssim \frac{1}{|\lambda|} + |x - y|.
	\end{align*}
    Furthermore, for any $\gamma \in [0, 1]$ and $|\lambda_1| \le |\lambda_2|$, the following Lipschitz bounds hold:
    \begin{align*}
        \abs{\mu_1(\lambda_2, x, y) - \mu_1(\lambda_1, x, y)} &\les  |\lambda_2 - \lambda_1|^{\gamma} \min\left(|\lambda_2|, \frac{1}{|x - y|}\right)^{1 - \gamma} \\
        \abs{\partial_\lambda\mu_1(\lambda_2, x, y) - \partial_\lambda\mu_1(\lambda_1, x, y)} &\les |\lambda_2 - \lambda_1|^\gamma\left(\frac{1}{|\lambda_1|} + |x - y|\right)^{\gamma}.
    \end{align*}
\end{lemma}
The first claim follows by dividing \eqref{eqn:mu0 bounds} by $\lambda$ and the first Lipschitz bounds follows by interpolation.  For the Lipschitz bound on the derivative, a case analysis as in the end of the proof of Lemma~\ref{lem: B inverse} is required.

To establish the first bound we may select $\ell=0+$ in the expansion of $(M^\pm)^{-1}$ in Proposition~\ref{prop:Minv eval}.    The Lipschitz bounds in Proposition~\ref{prop:Minv eval} restrict our choices for  $\gamma$ and $\ell$. We must select $\ell$ so that $\ell(1-\gamma)-\gamma>0$ to ensure integrability near zero.  To do so, we select $1\geq \ell=\frac{\gamma}{1-\gamma}+$,  which restricts $\gamma$ to $0 \leq \gamma < 1/2$. Under these conditions, $3+2(\gamma+\ell(1-\gamma))=3+4\gamma+$. 
To obtain an estimate that is integrable at infinity, 
we may select $\gamma=0+$ with $\ell=c\gamma$, for some small $c>0$.

\begin{proof}[Proof of Proposition~\ref{prop:zero not reg dispersive bounds}.]
	From \eqref{eqn:born series} and Proposition \ref{prop:Minv eval}, we have	
	\begin{align*}
		\mR_V^\pm(\lambda)= &\mR_0^\pm(\lambda)- \mR_0^\pm(\lambda) V \mR_0^\pm(\lambda)+ \mR_0^\pm (\lambda)V \mR_0^\pm(\lambda) V  \mR_0^\pm(\lambda)\\
		-&\mR_0^\pm(\lambda) V \mR_0^\pm(\lambda) v^*\left(-\frac{D_2}{\lambda} + \widetilde{\Gamma}_{0}^\pm+M_{-1,\ell}^\pm(\lambda)\right)  v \mR^{\pm}_0 (\lambda) V  \mR_0^\pm(\lambda).
	\end{align*}
	By Lemmas~\ref{lem:born series} and \ref{lem:low wtd Born}, we only need to consider the last term. Furthermore, by Proposition~\ref{prop:Minv eval} the operators $\widetilde \Gamma_0^\pm +M^\pm_{-1,\ell}(\lambda)$ satify the hypotheses of Lemmas~\ref{prop:perturbed bound} and \ref{lem:low wtd tail}, so we need only consider the contribution of $D_2$.

	Since $D_2$ is independent of the $\lambda$ and $\pm$, using \eqref{eqn:alg identity} we have
	\begin{multline*}
		\mR_0^+V\mR_0^+v^*	\frac{D_2}{\lambda} v\mR_0^+V\mR_0^+-\mR_0^-V\mR_0^-v^*		\frac{D_2}{\lambda} v\mR_0^-V\mR_0^- \\
		= \mu_1 V\mR_0^-v^*	D_2 v\mR_0^-V\mR_0^- + \mR_0^+V\mu_1v^*	D_2 v\mR_0^-V\mR_0^- \\
		+ \mR_0^+V\mR_0^+v^*D_2 v\mu_1V\mR_0^-  
		+ \mR_0^+V\mR_0^+v^*D_2 v\mR_0^+V\mu_1.
	\end{multline*}
	We choose to move the singular $\frac{1}{\lambda}$ to the difference $\mR_0^+-\mR_0^-$ to take advantage of the cancellation and smallness near zero and use the bounds in Lemma~\ref{lem:mu1 bounds} directly.

	We now consider the first bound in the claim.  
	We consider the contribution of the first term in the equation above; the argument for the other three is similar.  The proof follows the argument in Lemma~\ref{prop:perturbed bound} replacing one resolvent with $\mu_1(\lambda)$. The bounds on $\mu_1$ in Lemma~\ref{lem:mu1 bounds} allow us to  integrate by parts with no boundary terms
	\begin{multline*}
		-\int_{\R} e^{-it\lambda} \chi(\lambda)  [(\mR_0^+ - \mR_0^-)V\mR_0^-v^*	\frac{D_2}{\lambda} v\mR_0^-V\mR_0^-](\lambda)(x,y)\, d\lambda \\
 		=\frac{1}{it}\int_{\R}e^{-it\lambda}\p_\lambda[\mu_1V\mR_0^-v^*D_2 v\mR_0^-V\mR_0^-](\lambda)(x,y)\, d\lambda.
	\end{multline*}
	Then for $k_1, k_2, k_3, k_4 \in \{0, 1\}$ with $\sum k_j = 1$, the integrand is composed of sums of operators of the form
	\begin{align*}
		e^{-it\lambda}\partial_\lambda^{k_1}\mu_1V\partial_\lambda^{k_2}\mR_0^-v^*D_2v\partial_\lambda^{k_3}\mR_0^-V\partial_\lambda^{k_4}\mR_0^-.
	\end{align*}
	Now observe that $v\partial_\lambda^{k_3}\mR_0^-V\partial_\lambda^{k_4}\mR_0^-$ is $L^2$ by the boundedness of $D_2$ along with \eqref{eqn:4.3a} and \eqref{eqn:4.3b} in the proof of Lemma \ref{prop:perturbed bound}. Now observe that by \eqref{eq:resolv high derivs}, Lemmas \ref{lem:mu1 bounds}, \ref{lem:spatial estimates} and \ref{lem:GolVis}, on the support of $\chi(\lambda)$ we have
	\begin{align*}                                 
		\sup_{x\in\R^3}\left\|\big(\partial_\lambda^{k_1}\mu_1V\partial_\lambda^{k_2}\mR_0^-v^*\big)(\lambda)(x, \cdot)\right\|_{L^2}		\les 1.
	\end{align*}
	Then, by \eqref{eqn:4.3a} and the absolute boundedness of $D_2$, we have 
	\begin{align*}
		\sup_{x,y\in\R^3}\left|\int_{\R} e^{-it\lambda} \chi(\lambda) \lambda^{-1} [\mR_0^+V\mR_0^+v^*	D_2 v\mR_0^+V\mR_0^+-\mR_0^-V\mR_0^-v^*	D_2 v\mR_0^-V\mR_0^-](\lambda)(x,y)\, d\lambda\right| \les \frac{1}{|t|}.
	\end{align*}
	Here the spatial integrals are controlled as in the proof of Lemma \ref{prop:perturbed bound}.

	For the weighted bound,	we adapt the proof of Lemma~\ref{lem:low wtd tail} to account for the effect of $\mu_1$.  Since the smallness in $\lambda$ of the resolvents isn't used in the proof of Lemma~\ref{lem:low wtd tail}, the bound in \eqref{eqn:L2 iter Lips1} is valid if we replace one $\mR_0$ with $\mu_1$, the only new term that arises is the contribution of terms involving $\partial_\lambda \mu_1$.  We note that
    \begin{align*}
        &\big|[\partial_\lambda \mu_1(\lambda_2)-\partial_\lambda\mu_1(\lambda_1)]V\mR_0(\lambda_2)(x,z_2)\big|\\
		& \qquad \les |\lambda_2-\lambda_1|^{\gamma} \int_{\R^3} \bigg(\frac{1}{|\lambda_1|}+|x-z_1|\bigg)^{\gamma} \la z_1\ra^{-\delta}\bigg(\frac{1}{|z_1-z_2|^2}+\frac{|\lambda_2|}{|z_1-z_2|}\bigg) \, dz_1 \\
		& \qquad \les |\lambda_2-\lambda_1|^\gamma |\lambda_1|^{-\gamma}\la x\ra^{\gamma} \int_{\R^3}\la z_1\ra^{\gamma-\delta}\bigg(\frac{1}{|z_1-z_2|^2}+\frac{1}{|z_1-z_2|}\bigg) \, dz_1 \les |\lambda_2-\lambda_1|^\gamma |\lambda_1|^{-\gamma}\la x\ra^{\gamma}.
    \end{align*}
	Here we need $\delta>2+\gamma$ to apply Lemma~\ref{lem:spatial estimates}.  In particular, using Lemma~\ref{lem:GolVis} this shows that
	$$
		\|[\partial_\lambda \mu_1(\lambda_2)-\partial\mu_1(\lambda_1)]V\mR_0(\lambda_2)(x,\,\cdot\,)v^*(\,\cdot\,)\|_{L^{2}}\les |\lambda_2-\lambda_1|^\gamma |\lambda_1|^{-\gamma}\la x\ra^{\gamma}.
	$$
	From here, the proof follows exactly as in the regular case with the above bound used in place of  \eqref{eqn:L2 iter Lips2}.  The restrictions on $\gamma$ arise when using the Lipschitz bounds on $\partial_\lambda(M^{\pm})^{-1}$ since we need $\ell(1-\gamma)-\gamma-1>-1$ to ensure integrability near zero.
\end{proof}

\begin{rmk}
	
	The constraint on $\gamma$ in Theorem~\ref{thm:main} is an artifact of the proof.  It should be possible to prove similar results for $1/2 \leq \gamma\leq 1$ by using longer expansions.  That is, writing
	$$
		\mR_0^\pm(\lambda)
		= \mG_0 + \lambda \mG_1+i\lambda^2\mG_2^\pm+\lambda^3\mG_3^\pm +\mathcal E_3^\pm(\lambda,|x-y|)
	$$
	where $\mG_3^\pm=\pm \frac{|x-y|}{3}\alpha \cdot \hat{e} -\frac{|x-y|}{2}$ would allow for an error term of size $\lambda^3(\lambda|x-y|)^\ell$ for any $0\leq\ell \leq 1$.  Using this in expansions for $M^\pm$, $B^\pm$ and considering longer Neumann series expansions would allow for control of the error terms that avoids the bottleneck in the proof of Proposition~\ref{prop:zero not reg dispersive bounds}.  Since our proof allows for a time-integrable bound, we omit this approach for the sake of brevity.
	
\end{rmk}

\section{Threshold characterization}\label{sec:thresholds}

For completeness, we prove the claims in Definition~\ref{def:resonances}, which connect the existence of threshold eigenvalues with the $L^2$ kernel of the operator $T_0$.
The characterization of the threshold is similar to that of the massless two dimensional case, \cite{EGG2}, with roots in the massive case in \cite{egd,EGT} and Schr\"odinger equation \cite{JN,ES,eg2}. Though the lack of zero energy resonances simplifies many calculations.  Recall that $\mathcal H=D_0+V$.
\begin{lemma}\label{lem:S1 char}
	Assume that $|V(x)|\les \la x\ra^{-2-}$.
	If $\phi \in \Ker(T_0)$, then $\phi = U v \psi$ with $\psi$ a distributional solution to $\mathcal H\psi=0$ and $\psi\in L^2(\R^3)$, that is $\psi$ is an eigenfunction. Furthermore, $\psi \in L^p(\R^3)$ for all $p\geq 2$.
\end{lemma}

\begin{proof}
	Take $\phi \in \Ker(T_0)$ for $\phi \in L^2$, so 
    \begin{align*}
        0 
        = T_0\phi = U\phi + v\mG_{0}v^* \phi 
        = 0 \quad \implies \quad \phi = - Uv\mG_{0}v^* \phi.
    \end{align*}
	Define $\psi \coloneqq - \mG_{0}v^*\phi$, then $\phi = Uv\psi$. Now,   $\mathcal H = D_0+V = -i\alpha \cdot \nabla+V$,
    \begin{align*}
        \mathcal H\psi
        = (-i\alpha \cdot \nabla+V)\psi 
        = -i\alpha \cdot \nabla \psi + v^*Uv\psi
		= -i\alpha \cdot \nabla(\mG_{0}v^*\phi) + v^*\phi .
    \end{align*}
	Here, recalling that $\mG_0=-i\alpha \cdot \nabla G_0$ where $G_0(x,y) = (4\pi |x-y|)^{-1} = (-\Delta)^{-1}(x,y)$, we have
    \begin{align*}
        -i\alpha \cdot \nabla(\mG_{0}v^*\phi) 
        = -i\alpha \cdot \nabla(-i\alpha \cdot \nabla G_0 v^*\phi) = \Delta (-\Delta)^{-1}v^*\phi 
        = - v^*\phi 
    \end{align*}
	distributionally. So,
    \begin{align*}
        \mathcal H\psi 
        = -i\alpha \cdot \nabla(\mG_{0}v^*\phi) + v^*\phi 
        = - v^*\phi + v^*\phi 
        = 0.
    \end{align*}
	That is, if $\phi \in \Ker(T_0)$ we have $\mathcal H\psi = 0$ in the sense of distributions. Now, to show that $\psi \in L^2$, we note that $\psi = -\mG_{0}v^*\phi$ with $\phi \in L^2$. We can dominate the kernel of $\mG_0$ as follows: $|\mG_{0}| \les \mathcal I_1$ where $\mathcal I_1$ is the fractional integral operator with integral kernel $\mathcal{I}_1(x,y) = c|x-y|^{-2}$. By Lemma 2.3 in \cite{Jen} $\mathcal I_1 \colon L^{2, \sigma} \to L^2$ provided $\sigma > 1$.  If we assume $|V(x)| \les \la x \ra^{-\delta}$ for some $\delta > 2$, then $v^*\phi \in L^{2,1+}$, and we conclude that $\psi \in L^2(\R^3)$.
	
	Further, by the Hardy-Littlewood-Sobolev inequality, $\mathcal{I}_1 \colon L^2(\R^3) \to L^6(\R^3)$, hence we have $\psi \in L^6(\R^3)$. Using that $\psi = -\mG_0v^*\phi$ and $\phi = Uv\psi$, we have $\psi = -\mG_0V\psi$
	\begin{align*}
		|\psi(x)|
        \les |\mG_0 V \psi(x)| 
        \les \int_{\R^3} \frac{|V(y)\psi(y)|}{|x-y|^2} \, dy 
        \leq \| V(y)|x-y|^{-2}\|_{L^{\frac{6}{5}}_y} \|\psi\|_6 
        \les 1.
	\end{align*}
	The last inequality holds uniformly in $x\in \R^3$ provided $|V(y)| \les \la y\ra^{-\delta}$ for some $\delta >1/2$ by Lemma~\ref{lem:GolVis}, hence $\psi \in L^\infty$.
\end{proof}

This argument also shows that zero energy resonances do not exist. If $\psi\in L^{2,-\frac12-}$ solves $\mathcal H\psi=0$, the same argument shows we can bootstrap $\psi\in L^2$, hence $\psi$ is an eigenfunction.

We define $S_1$ to be the orthogonal projection onto the kernel of $T$. By standard arguments, $S_1$ is a finite rank projection, see Definition~\ref{def:resonances} above.
\begin{lemma}\label{lem:efn char}
	Assume that $|V(x)|\les \la x\ra^{-2-}$.
	If $\mathcal H\psi = 0$ with $\psi \in L^2 $, then $\phi = Uv\psi \in S_1L^2$, i.e. $T_0\phi = 0$.
\end{lemma}

\begin{proof}
	If $0 = \mathcal H\psi$, then $i\alpha \cdot \nabla \psi = V\psi = v^* \phi$.  To show that $\psi = -\mG_{0}v^*\phi$, noting that $\phi = Uv\psi \in L^2 \subseteq L^1_{loc}$, we have that $v^*\phi \in L^1$. Recalling that $\mG_{0} = -i\alpha \cdot \nabla G_0$, so that $\Delta(-i\alpha \cdot \nabla \mG_0)v^* \phi = -v^* \phi$ in the sense of distributions, we see that
	\begin{align*}
		-i\alpha \cdot \nabla \big[\psi +\mG_{0}v^* \phi \big]
		& = -i\alpha \cdot \nabla \psi + \Delta G_0 v^* \phi 
        = v^*\phi-v^*\phi 
        = 0.
	\end{align*}
	This shows that $\psi + \mG_0v^*\phi$ is annihilated by the gradient, we must have that
	$\psi + \mG_0v^*\phi = (c_1,c_2,c_3,c_4)^T$ is a constant vector. But, $\psi \in L^2$ and the argument in the Lemma above shows that $\mG_{0}v^*\phi \in L^2$. Hence, $(c_1,c_2,c_3,c_4)^T \in L^2$, which necessitates that $c_j = 0$ for each $j$. Thus, $\psi = -\mG_{0}v^*\phi$ as desired. 
	
	Noting that $\phi = Uv\psi = -Uv\mG_{0}v^* \phi$, and recalling that $T_0 = U + v\mG_{0}v^*$, $i\alpha \cdot\nabla \psi = V\psi = v^*\phi$, we see that
    \begin{align*}
        T_0\phi =U\phi+v\mG_0v^*\phi
        =v\psi+v\mG_{0}V\psi
        =v\psi+v\mG_{0} v^* \phi
        =v\psi-v\psi
        =0.
    \end{align*}
	Hence $\phi \in S_1L^2$ as desired.
\end{proof}

Now, we  show that $S_1 v \mG_{1}v^* S_1$ is always  invertible on $S_1L^2$.

\begin{lemma}\label{lem:invertibility}
	We have the identity
	\begin{align}\label{eqn:G0G1 identity}
		\la \mG_{0}v^*\phi, \mG_{0}v^*\phi \ra = -\la v^*\phi, \mG_{1}v^*\phi \ra.
	\end{align}
	Furthermore, the kernel of $S_1v\mG_1v^* S_1$ is trivial.
\end{lemma}

\begin{proof}
	We first note that $\mG_{0}=-i\alpha \cdot \nabla G_0$, where $G_0=(-\Delta)^{-1}$, moving to the Fourier side we see:
	\begin{align*}
		\la \mG_0 v^*\phi, \mG_0 v^*\phi\ra = \int_{\R^3} \frac{1}{|\xi|^4} \langle A(\xi)
		\widehat{v^*\phi},  A(\xi)
		\widehat{v^*\phi}\rangle_{\C^4}\, d\xi
	\end{align*}
	where
    \begin{align*}
        A(\xi) =
        \begin{pmatrix} 0 & 0 & \xi_3 & -i\xi_1 + \xi_2 \\
		0 & 0 & i\xi_1 + \xi_2 & -\xi_3\\
		\xi_3 & -i\xi_1 + \xi_2 & 0 & 0\\
		i\xi_1 + \xi_2 & -\xi_3 & 0 & 0 \end{pmatrix}.
    \end{align*}
	We note that $A(\xi)$ is self-adjoint and
	$A^*(\xi)A(\xi)=|\xi|^2 I_{4\times 4}$. From here, we see that
    \begin{align*}
        \la \mG_0 v^*\phi, \mG_0 v^*\phi\ra 
        = \int_{\R^3} \frac{1}{|\xi|^2} \langle 
		\widehat{v^*\phi}, \widehat{v^*\phi}\rangle_{\C^4}\, d\xi.
    \end{align*}
	On the other hand, we recall the Schr\"{o}dinger resolvent $R_0(\lambda^2)$ has Fourier transform $(|\xi|^2 - \lambda^2)^{-1}$.  Evaluating the Schr\"{o}dinger resolvent at $-\lambda^2$ for any $\lambda > 0$   in the resolvent set, then one has
	$\mathcal F(R_0(-\lambda^2)) = (|\xi|^2 + \lambda^2)^{-1}$.  Using the expansion that $R_0(-\lambda^2)=G_0 + O(\lambda^{0+})$ as $\lambda \to 0$. Recalling that $\mG_1 = G_0 I_{4 \times 4}$, we have (again going to the Fourier side)
	\begin{align*}
		\la v^*\phi, \mG_1 v^*\phi \ra 
        = \lim_{\lambda \to 0} \la v^*\phi, R_0(-\lambda^2) I_{4\times 4}v^*\phi \ra 
		= \lim_{\lambda\to 0} \int_{\R^3} \frac{1}{|\xi|^2 + \lambda^2}\langle 
		\widehat{v^*\phi}, \widehat{v^*\phi}\rangle_{\C^4} \, d\xi.
	\end{align*}
	Applying the dominated convergence theorem, we bring the limit inside the integral to see
    \begin{align*}
        \la v^*\phi, \mG_1 v^*\phi \ra 
        = \int_{\R^3} \frac{1}{|\xi|^2} \langle 
		\widehat{v^*\phi}, \widehat{v^*\phi}\rangle_{\C^4}\, d\xi
        = \la \mG_0 v^*\phi, \mG_0 v^*\phi\ra,
    \end{align*}
	as claimed.
	
	Now, take $\phi \in S_1L^2$ in the kernel of $S_1v\mG_1 v^* S_1$. By Lemmas \ref{lem:S1 char} and \ref{lem:efn char} we have $\psi = -\mG_0v^*\phi$ and $\phi = Uv\psi$. Since $S_1v\mG_1v^*S_1\phi = 0$ we have:
	\begin{align*}
		0 = \la \phi, S_1v\mG_1v^*S_1\phi\rangle 
        = \la v^*\phi, \mG_1 v^*\phi\ra 
        = \la \mG_0 v^*\phi, \mG_0 v^*\phi\ra
        = {\|\psi\|}_{L^2}^2.
	\end{align*}
	Hence $\psi = 0$, and $Uv\psi = \phi=0$.
\end{proof}
This shows that $S_1v\mG_1 v^*S_1$ is invertible on $S_1L^2$ as desired.
It follows that
\begin{align*}
    P_0
    = \mG_{0}v S_1 [S_1 v\mG_{1} v^* S_1]^{-1} S_1 v^* \mG_{0} =\mG_{0}v D_2 v^* \mG_{0}
\end{align*}
The proof of this is follows the argument of Lemma~7.10 in \cite{egd}, which proved this in the massive two-dimensional case.  We do not use this projection, so we leave the proof to the interested reader.

\section{High Energy}\label{sec:hi}

Finally, we control the high energy portion of the evolution to complete the proofs of Theorems~\ref{thm:full results} and \ref{thm:weak results}.  Here one cannot use the expansions for $\mR_V^\pm$ developed for the low energy expansions.  Instead, we use the limiting absorption principle, \cite{EGG2}:
\begin{align}\label{eqn:LAP}
	\sup_{\lambda>0} \|\partial_\lambda^k \mR_V^\pm(\lambda)\|_{L^{2,\sigma+k}\to L^{2,-\sigma-k}} \les 1, \qquad \sigma>\f12, \qquad k=0,1,2.
\end{align}
This requires only that $|V(x)|\les \la x\ra^{-1-}$ and that $V$ has continuous entries.  For high energy one has a sharper control on decay of the potential, though it requires continuity of the potential.

Here we cannot use the Lipschitz continuity argument invoked in the low energy regime, but instead proceed via integrating by parts in the Stone's formula, \eqref{eqn:stone}.  We also selectively iterate the resolvent identity by decomposing $\mR_0$ into $\mR_L$ and $\mR_H$ as in the proof of Theorem~\ref{thm:weak decay}, here with an eye on minimizing the growth in the spectral parameter $\lambda$ rather than  to limit the needed decay on $V$.
\begin{prop}\label{prop:hi energy}
	
	Let $|V(x)| \les \la x\ra^{-\delta}$ for some $\delta>1$ with continuous entries.  Then
	$$
	\sup_{x,y\in \R^3} \bigg| \int_{\R} e^{-it\lambda} \widetilde \chi(\lambda) \la \lambda\ra^{-3-} [\mR_V^+-\mR_V^-](\lambda)(x,y)\, d\lambda \bigg| \les 1.
	$$
	If $\delta>2$, then
	$$
		\sup_{x,y\in \R^3} \bigg| \int_{\R} e^{-it\lambda} \widetilde \chi(\lambda) \la \lambda\ra^{-3-} [\mR_V^+-\mR_V^-](\lambda)(x,y)\, d\lambda \bigg| \les \frac{1}{|t|}.
	$$
	Further, if $\delta>3$ we have the weighted bound
	$$
		\bigg| \int_{\R} e^{-it\lambda} \widetilde \chi(\lambda) \la \lambda\ra^{-3-} [\mR_V^+-\mR_V^-](\lambda)(x,y)\, d\lambda \bigg| \les \frac{\la x\ra \la y\ra}{|t|^2}.
	$$
	
\end{prop}

 When $\lambda$ is bounded away from zero, we recall \eqref{eqn:RL RH defn} and note that
By the expansions developed in Lemma~\ref{lem:R0 expansions}, we have (for $k=0,1,2$)
\begin{align}\label{eq:RL RH}
	|\partial_\lambda^k \mR_L^\pm(\lambda)(x,y)|\les \frac{1}{|\lambda|^k|x-y|^{2}}, \qquad 	|\partial_\lambda^k \mR_H^\pm(\lambda)(x,y)|\les |\lambda|\, |x-y|^{k-1}.
\end{align}
In particular, there is only growth in $\lambda$ when $\mR_H^\pm$ appears, while $\mR_L^\pm$ is more singular in the spatial variables which necessitates iteration of the Born series.  
A straight forward computation using Lemma~\ref{lem:GolVis} shows that for $\sigma > k + 1/2$ we have
\begin{align}\label{eqn:RH wtd L2}
	\| \partial_\lambda^k \mR_H^\pm(\lambda)(x,\,\cdot\,)\|_{L^{2,-\sigma}}\les |\lambda| \la x\ra^{k-1}.
\end{align}
In particular, this bound is uniform when $k=0,1$.  While $\mR_L^\pm$ and its derivatives are not locally $L^2$.
Accordingly, we write (omitting the $\pm$ for the moment)
\begin{align}\label{eqn:hi born}
	\mR_V =\mR_0 -\mR_0  V \mR_0 +\mR_0  V\mR_V  V\mR_0 .
\end{align}
The first term is controlled by Theorem~\ref{thm:free} and Corollary~\ref{cor:free weak}. For the second term we need to utilize the difference between the `$+$' and `$-$' resolvents in the Stone's formula, while the third term requires more careful and selective iteration.  We note that the factor of $\la \lambda \ra^{-3-}$ is needed here since each iteration of $\mR_0$ or $\mR_H$ contributes a growth of size $\lambda$ in the spectral parameter.  To ensure the $\lambda$ integral converges at infinity, we must control a growth of size $|\lambda|^2$ and be integrable at infinity.  We prove Proposition~\ref{prop:hi energy} in a series of lemmas.
\begin{lemma}\label{lem:hi born1}
	Let $|V(x)|\les \la x\ra^{-\delta}$ for some $\delta>1$.  Then
	\begin{align*}
		\sup_{x,y\in\R^3} \bigg| \int_{\R} e^{-it\lambda} \widetilde \chi(\lambda) \la \lambda \ra^{-3-} [\mR_0^+  V \mR_0^+-\mR_0^-  V \mR_0^-](\lambda)(x,y)\, d\lambda|\les 1.
	\end{align*}
	If $\delta>2$,
    \begin{align*}
       \sup_{x,y\in\R^3} \bigg| \int_{\R} e^{-it\lambda} \widetilde \chi(\lambda) \la \lambda \ra^{-3-} [\mR_0^+  V \mR_0^+-\mR_0^-  V \mR_0^-](\lambda)(x,y)\, d\lambda\bigg|\les |t|^{-1}.
    \end{align*}
	If $\delta>3$,
	\begin{align*}
		\bigg| \int_{\R} e^{-it\lambda} \widetilde \chi(\lambda) \la \lambda \ra^{-3-} [\mR_0^+  V \mR_0^+-\mR_0^-  V \mR_0^-](\lambda)(x,y)\, d\lambda\bigg|\les \frac{\la x\ra \la y \ra}{|t|^2}.
	\end{align*}
\end{lemma}
\begin{proof}
	By \eqref{eqn:alg identity} and symmetry, it suffices to bound $[\mR_0^+-\mR_0^-]V\mR_0^+$. For the first claim we write the resolvents on the left as $\mR_0=\mR_L+\mR_H$ and consider cases. For the contribution of $\mR_L$ on the left, we note that the difference of resolvents satisfies both \eqref{eqn:mu0 bounds} as well as the bounds for $\mR_L$ in \eqref{eq:RL RH} by the triangle inequality. As a consequence, we have
	$|[\mR_L^+-\mR_L^-](\lambda)(x,y)|\les \min(|\lambda|^2,|x-y|^{-2})$, which then implies
	$$
	|(\mR_L^+-\mR_L^-)(\lambda)(x,z)V(z)\mR_0^+(\lambda)(z,y)|\les \la z\ra^{-\delta} \bigg(\frac{|\lambda|^{0+}}{|x-z|^{2-}|z-y|}+\frac{|\lambda|^{1+}}{|x-z|^{1-}|z-y|^2}
	\bigg).
	$$
	On the right side we wrote $\mR_0=\mR_L+\mR_H$ and used \eqref{eqn:RLRH bds}. 	
	If $\mR_H$ is on the left, we do not use any cancellation between `$+$' and `$-$' resolvents but note that we may multiply by $|\lambda|\,|x-z|$ as needed to ensure the spatial integrals are bounded uniformly in $x$ and $y$, so
	$$
	|\mR_H^\pm(\lambda)(x,z)V(z)\mR_0^+(\lambda)(z,y)|\les \la z\ra^{-\delta} \bigg(\frac{|\lambda|^{1+}}{|x-z|^{1-}|z-y|^2}+\frac{|\lambda|^{0+}}{|x-z|^{2-}|z-y| }
	\bigg).
	$$
	In any case, by applying Lemma~\ref{lem:spatial estimates} with $\delta>0$, we see that the spatial integrals are bounded uniformly in $x,y$.  Hence, we have
	$$
	\sup_{x,y\in\R^3} \bigg| \int_{\R} e^{-it\lambda} \widetilde \chi(\lambda) \la \lambda \ra^{-3-} [\mR_0^+  V \mR_0^+-\mR_0^-  V \mR_0^-](\lambda)(x,y)\, d\lambda\bigg|\les  \sup_{x,y\in\R^3} \int_{\R} \la \lambda \ra^{-2-}\, d\lambda \les 1.
	$$
	
	We consider the second bound. By \eqref{eqn:mu0 bounds} and \eqref{eqn:R0 bounds}, the support of the cut-off and the decay of $\la \lambda \ra^{-3-}$, there are no boundary terms when we integrate by parts.  We note that differentiation of the cut-off and $\la \lambda\ra^{-3-}$ is comparable to division by $\lambda$. By the triangle inequality we need to bound
	\begin{align*}
        &\frac{1}{|t|} \int_\R \big|\p_\lambda  \big[ \widetilde{\chi}(\lambda) \ang{\lambda}^{-3-} \big(\mR_0^+ - \mR_0^-\big)V\mR_0^-(\lambda)(x,y)\big]\big|d\lambda.
    \end{align*}
	Using the bounds in \eqref{eqn:mu0 bounds} and \eqref{eqn:R0 bounds}, the above integral is dominated by
	\begin{multline*}
        \frac{1}{|t|}\int_\R \int_{\R^3}\widetilde{\chi}(\lambda) \ang{\lambda}^{-1-}\la z\ra^{-\delta} \left(\frac{1}{|x-z|^{1-}|z-y|^2} +\frac{1}{|z - y|^2}+\frac{1}{|x - z|^2}+\frac{1}{|x - z|}\right)dz \, d\lambda \nonumber\\
         \les \frac{1}{|t|}\int_\R   \ang{\lambda}^{-1-}  d\lambda
        \les \frac{1}{|t|} ,\nonumber
	\end{multline*}
    where we require $\delta > 2$ to apply Lemma~\ref{lem:spatial estimates}.  In the case when the derivatives don't act on a resolvent, we interpolate between the two bounds for $\mu(\lambda)$ in \eqref{eqn:mu0 bounds} to bound the difference of resolvents by $|\lambda|^{1+}|x-z|^{-1+}$ to avoid the logarithmic singularity in the spatial integral.
    
    For the final bound we may integrate by parts a second time without  boundary terms.  Ignoring when the derivative acts on the first two terms, whose contribution is bounded by $|t|^{-1}$ using the argument above, we use \eqref{eqn:mu0 bounds} and \eqref{eqn:R0 bounds} to control
    \begin{align*}
        &\frac{1}{t^2}\int_{\R}\bigg| \widetilde{\chi}(\lambda)\ang{\lambda}^{-3-}\p_\lambda^2\big[(\mR_0^+ - \mR_0^-)V\mR_0^-(\lambda)(x,y)\big]\, d\lambda \bigg| \\
        & \qquad \qquad \les \frac{1}{t^2} \int_\R  \ang{\lambda}^{-1-} \int_{\R^3}\la z\ra^{-\delta}\bigg(\frac{\la x\ra \la z\ra+|x-z|}{|x-z|^2}+1+\frac{\la z\ra \la y\ra}{|z-y|} \bigg)
        dz\, d\lambda \les \frac{\ang{x}\ang{y}}{t^2},
    \end{align*}
    where we used $|x-z|\les \la x\ra \la z\ra$ and require $\delta > 3$ to apply Lemma~\ref{lem:spatial estimates}. 
\end{proof}

\begin{lemma}\label{lem:hi tail}
    Let $|V(x)|\les \la x\ra^{-\delta}$ for some $\delta>1$ with continuous entries. Then, 
    \begin{align*}
    	\sup_{x,y\in\R^3} \bigg| \int_{\R} e^{-it\lambda} \widetilde \chi(\lambda) \la \lambda \ra^{-3-} \mR_0 ^\pm V\mR_V^\pm  V\mR_0^\pm(\lambda)(x,y)\, d\lambda\bigg|\les 1.
    \end{align*}
    If $\delta>2$ we have
    \begin{align*}
        \sup_{x,y\in\R^3} \bigg| \int_{\R} e^{-it\lambda} \widetilde \chi(\lambda) \la \lambda \ra^{-3-} \mR_0 ^\pm V\mR_V^\pm  V\mR_0^\pm (\lambda)(x,y)\, d\lambda\bigg|\les |t|^{-1}.
    \end{align*}
	If $\delta>3$ we have
	\begin{align*}
		\bigg| \int_{\R} e^{-it\lambda} \widetilde \chi(\lambda) \la \lambda \ra^{-3-} \mR_0 ^\pm V\mR_V^\pm  V\mR_0^\pm (\lambda)(x,y)\, d\lambda\bigg|\les \frac{\la x\ra \la y \ra}{|t|^2}.
	\end{align*}
\end{lemma}

\begin{proof}
	In this case we do not utilize the difference between the `$+$' and `$-$' resolvents, but do selectively iterate. Accordingly, we suppress the $\pm$ notation and write
	\begin{multline}\label{eqn:hi tail iterate}
        \mR_0 V\mR_V V\mR_0 = \mR_H V\mR_V V\mR_H\\
        +	\mR_L V\mR_0 V\mR_V V\mR_H 
        +\mR_H V\mR_V V\mR_0 V \mR_L 
        +\mR_L V\mR_0 V\mR_V V\mR_0 V\mR_L.
	\end{multline}
	Using \eqref{eq:RL RH} and \eqref{eqn:R0 bounds}, with $k_j\in \{0,1,2\}$ and $k_1+k_2=k$, we see that 
   \begin{align*}
        \big|\partial_\lambda^k \big(\mR_L(\lambda)(x,z) V(z)\mR_0(\lambda)(z,y)\big)\big| 
        & \les \int_{\R^3} \frac{\la z\ra^{-\delta}}{|\lambda|^{k_1}|x-z|^2}\bigg(\frac{1}{|z-y|^2}+\frac{|\lambda|}{|z-y|} \bigg)|z-y|^{k_2}\\
        & \les \la \lambda\ra \int_{\R^3} \frac{\la z\ra^{-\delta}}{|x-z|^2|z-y|^2}(1+|z-y|^k)\, dz.
    \end{align*}
	Applying Lemma~\ref{lem:spatial estimates}, if $k=0,1$ we bound by
	$\la \lambda\ra (1+ |x-y|^{-1})$ provided $\delta>1$.  Applying Lemma~\ref{lem:spatial estimates} shows that
	\begin{align}\label{eqn:RL k01}
		\sup_{x\in \R^3} \| \partial_\lambda^k \big(\mR_L V\mR_0(\lambda)(x,\cdot)\|_{L^{2,-\sigma}}\les \la \lambda \ra \qquad k=0,1,
	\end{align}
	provided $\sigma>k + 1/2$ and $\delta>2$. When $k = 2$ we see that
	\begin{align}\label{eqn:RL k2}
		\| \partial_\lambda^2 \big(\mR_L V\mR_0(\lambda)(x,\,\cdot\,)\|_{L^{2,-\sigma}}\les \la \lambda \ra \la x\ra  ,
	\end{align}
	provided $\sigma > 3/2$ and $\delta > 2$.
	
	Using \eqref{eqn:hi tail iterate}, we may express the integral we need to bound as 
    \begin{equation}\label{eqn:hi iterated goal}
        \int_{\R}e^{-it\lambda}\widetilde{\chi}(\lambda)\ang{\lambda}^{-3-}\Gamma_{1,x}(\lambda)V \mR_V(\lambda)V \Gamma_{2,y}(\lambda)(x,y) d\lambda
    \end{equation}
    where \eqref{eqn:RH wtd L2}, \eqref{eqn:RL k01} and \eqref{eqn:RL k2} show that (for $j = 1, 2$)
    \begin{align}\label{eqn:Gamma hi}
    	\sup_{x\in \R^3}\|\partial_\lambda^k \Gamma_{j, x}(\lambda)\|_{L^{2, -\sigma}}\les \la \lambda\ra, \qquad k=0,1, \qquad
    	\|\partial_\lambda^2 \Gamma_{j, x}(\lambda)\|_{L^{2, -\sigma}}\les \la \lambda\ra \la x\ra,
    \end{align}
    provided $\sigma > k + 1/2$ and $\delta>1+k$ for $k = 0, 1, 2$. The bounds hold for $\Gamma_{j,y}$  as well, and remain valid for the adjoint operators since $V$ is self-adjoint and $(\mR_0^\pm)^*=\mR_0^{\mp}$.  
    
    The first claim follows by writing the operators in the integrand in terms of the $L^2$ inner product, \eqref{eqn:Gamma hi}, and the limiting absorption principle \eqref{eqn:LAP}.  Taking $\sigma=\frac{1}{2}+$ and $\delta>1$, we have
    \begin{multline*}
    	|\eqref{eqn:hi iterated goal}|
    	\les \bigg|\int_{\R} \la \lambda\ra^{-3-} \big\la \Gamma_{1,x}^*(\lambda),V\mR_V(\lambda) \Gamma_{2,y}(\lambda) 
    	\big\ra_{L^2}\, d\lambda \bigg| \\
    	\les \int_{\R}  \la \lambda\ra^{-3-} \|\Gamma_{1,x}^*(\lambda)\|_{L^{2,-\sigma}} \|V\mR_V(\lambda) \Gamma_{2,y}(\lambda) \|_{L^{2,\sigma}}\, d\lambda
    	\\
    	\les \int_{\R} \int_{\R^3} \la \lambda\ra^{-3-}\|\Gamma_{1,x}\|_{L^{-\sigma}} \|V\|_{L^{2,-\sigma}\to L^{2,\sigma}}
    	\|\mR_V\|_{L^{2,\sigma}\to L^{2,-\sigma}}
    	\|V\|_{L^{2,-\sigma}\to L^{2,\sigma}}\|\Gamma_{2,y}\|_{L^{-\sigma}}\, d\lambda\\
    	\les \int_{\R} \la \lambda\ra ^{-1-}\, d\lambda \les 1.
    \end{multline*}
	This bound holds uniformly in $x,y\in \R^3$.
    For the second claim, we integrate by parts.  The bounds in \eqref{eqn:Gamma hi} above and the decay of $\la \lambda\ra^{-3-}$ ensure there are no boundary terms.
    \begin{align*}
    	|\eqref{eqn:hi iterated goal}| \les \frac{1}{|t|} \int_{\R} \int_{\R^3} \bigg|
    	\partial_\lambda^{k_1}\widetilde{\chi}(\lambda)\partial_\lambda^{k_2}\ang{\lambda}^{-3-}\partial_\lambda^{k_3}\Gamma_{1,x}(\lambda)V \partial_\lambda^{k_4}\mR_V(\lambda)V \partial_\lambda^{k_5}\Gamma_{2,x}(\lambda)(x,y) d\lambda
    	\bigg| 
    \end{align*}
    where $k_j\in \{0,1\}$ and $\sum k_j=1$.
    The operators in the integrand may be controlled as in the first claim using \eqref{eqn:Gamma hi} and the limiting absorption principle \eqref{eqn:LAP}  as follows
    \begin{multline*}
    	\|\partial_\lambda^{k_3}\Gamma_{1,x}\|_{L^{2,-(\frac{1}{2}+k_3)-}} \|V\|_{L^{2,-(\frac{1}{2}+k_4)-}\to L^{2,\frac{1}{2}+k_3+}}
    	\|\partial_\lambda^{k_4}\mR_V(\lambda)\|_{L^{2,-(\frac12+k_4)-}}\\
    	\|V\|_{L^{2,-(\frac{1}{2}+k_5)-}\to L^{2,\frac{1}{2}+k_4+}}
    	\|\partial_\lambda^{k_5}\Gamma_{2,x}(\lambda)\|_{L^{2,-(\frac{1}{2}+k_5)-}} \les \la \lambda\ra^{2}.
    \end{multline*}
    The decay on $V$ is needed to map between weighted spaces, one needs $\delta>2$ to ensure multiplication by $V$ maps $L^{2,-\frac{1}{2}-}\to L^{2,\frac{3}{2}+}$.  Since only one $k_j$ can be nonzero, this suffices to control the spatial integrals and see that
    $$
    	\sup_{x,y \in \R^3}|\eqref{eqn:hi iterated goal}| \les \frac{1}{|t|} \int_{\R} \int_{\R^3} \la \lambda\ra^{-1-}\, d\lambda\les \frac{1}{|t|}.
    $$  
    The final bound follows similarly by integrating by parts and noting that $\sum k_j=2$.  In this case again using \eqref{eqn:Gamma hi} and \eqref{eqn:LAP} we have
    \begin{multline*}
    	\|\partial_\lambda^{k_3}\Gamma_{1,x}\|_{L^{2,-(\frac{1}{2}+k_3)-}} \|V\|_{L^{2,-(\frac{1}{2}+k_4)-}\to L^{2,\frac{1}{2}+k_3+}}
    	\|\partial_\lambda^{k_4}\mR_V(\lambda)\|_{L^{2,-(\frac12+k_4)-}}\\
    	\|V\|_{L^{2,-(\frac{1}{2}+k_5)-}\to L^{2,\frac{1}{2}+k_4+}}
    	\|\partial_\lambda^{k_5}\Gamma_{2,x}(\lambda)\|_{L^{2,-(\frac{1}{2}+k_5)-}} \les \la \lambda\ra^{2}\la x\ra \la y\ra.
    \end{multline*}
	Here, one needs $\delta>3$ since $\max(|k_j-k_i|)=2$, the mapping between weighted spaces must map between spaces of the form $L^{2,\sigma-}\to L^{2,\sigma+3}$.
    We have
    $$
   		|\eqref{eqn:hi iterated goal}| \les \frac{1}{|t|} \int_{\R} \int_{\R^3} \la \lambda\ra^{-1-}\, d\lambda\les \frac{\la x\ra \la y\ra}{|t|^2}.
    $$
\end{proof}   
Proposition~\ref{prop:hi energy} follows expanding $\mR_V$ as in \eqref{eqn:hi born} and applying Theorem~\ref{thm:free}, Corollary~\ref{cor:free weak}, Lemmas~\ref{lem:hi born1} and \ref{lem:hi tail} to control each term individually.

\section*{Statements and Declarations}
	Declarations of interest: none.  The first author was partially supported by Simons Foundation Grant 511825.  No data was used for the research described in this article.


\begin{thebibliography}{9}
 

\bibitem{agmon} Agmon, S. {\em Spectral properties of Schr\"odinger
operators and scattering theory.}
Ann.\ Scuola Norm.\ Sup.\ Pisa Cl.\ Sci. (4) 2 (1975), no.~2, 151--218.
 

\bibitem{MY}
Arai, M., and  Yamada, O. \emph{Essential selfadjointness and invariance of the essential spectrum for Dirac operators.}
Publ. Res. Inst. Math. Sci. 18 (1982), no. 3, 973--985. 

 
 
\bibitem{BG1}
A. Berthier  and V. Georgescu, \emph{On the point spectrum of Dirac operators.} J. Funct. Anal. 71 (1987), no. 2, 309--338. 

\bibitem{BH3}
Bejenaru, I., and Herr, S.  \emph{The cubic Dirac equation: small initial data in $H^{1}(\R^3)$.}
Comm. Math. Phys. 335 (2015), no. 1, 43--82. 

\bibitem{BH}
Bejenaru, I., and Herr, S.  \emph{The cubic Dirac equation: small initial data in $H^{1/2}(\R^2)$.}  Commun. Math. Phys. 343 (2016), 515--562.


\bibitem{Bouss1}
Boussa\"id, N. \emph{Stable directions for small nonlinear Dirac standing waves.} Comm. Math. Phys. 268 (2006), no. 3, 757--817.

\bibitem{BC}
Boussa\"id, N., and Comech, A. \emph{On spectral stability of the nonlinear Dirac equation.} J. Funct. Anal. 271 (2016), no. 6, 1462--1524.

\bibitem{BC2}
Boussa\"id, N. , and Comech, A. 
\emph{Spectral stability of small amplitude solitary waves of the Dirac equation with the Soler-type nonlinearity,}  J. Funct. Anal., 277 (2019), no. 12, 108289. 


\bibitem{BCbook}
Boussa\"id, N. , and Comech, A
\emph{Nonlinear Dirac equation.
	Spectral stability of solitary waves}. Mathematical Surveys and Monographs, 244. American Mathematical Society, Providence, RI, (2019), vi+297 pp. ISBN: 978-1-4704-4395-5 
 
\bibitem{BG}
Boussa\"id, N., and Golenia, S. \emph{Limiting absorption principle for some long range perturbations of Dirac systems at threshold energies.} Comm. Math. Phys. 299 (2010), no. 3, 677--708. 



\bibitem{CSZ}
Cacciafesta, F., S\'er\'e, \'E., and Zhang, J. \emph{Asymptotic estimates for the wave functions of the Dirac-Coulomb operator and applications.}
Comm. Partial Differential Equations 48 (2023), no.3, 355--385.

\bibitem{CGetal}
Carey, A.,  Gesztesy, F.,  Kaad, J.,  Levitina, G.,  Nichols, R.,  Potapov, D., and  Sukochev, F.,
\emph{On the Global Limiting Absorption Principle for Massless Dirac Operators}.  Ann. Henri Poincar\'e 19, No. 7, 1993–-2019 (2018).



 \bibitem{CTS} Comech, A., Phan, T., and Stefanov, A.  \emph{Asymptotic stability of solitary waves in generalized Gross-Neveu model}. Ann. Inst. H. Poincar\'e C Anal. Non Lin\'eaire.A 34 (2017), no. 1, 157--196.
 


\bibitem{DF}
D'Ancona, P., and Fanelli, L. \emph{Decay estimates for the wave and Dirac equations with a magnetic potential}. Comm. Pure Appl. Math. 60 (2007), no. 3, 357--392. 

\bibitem{Dan} Danesi, E. \emph{Strichartz estimates for the 2D and 3D massless Dirac-Coulomb equations and applications.}
J. Funct. Anal. 286 (2024), no.3, Paper No. 110251.



\bibitem{EGG}
Erdo\smash{\u{g}}an, M.~B., Goldberg, M, and Green, W.~R.  \emph{Limiting absorption principle and Strichartz estimates for Dirac operators in two and higher dimensions.} Comm. Math. Phys. 367 (2019), no. 1, 241--263.

\bibitem{EGG2}
Erdo\smash{\u{g}}an, M.~B., Goldberg, M, and Green, W.~R.  \emph{The Massless Dirac Equation in Two Dimensions: Zero-Energy Obstructions and Dispersive Estimates},  J. Spectr. Theory 11 (2021), no. 3, 935--979. 


\bibitem{EG1}
Erdo\smash{\u{g}}an, M.~B., and Green, W.~R. \emph{Dispersive estimates for the Schrodinger equation for $C^{\frac{n-3}{2}}$ potentials in odd dimensions}. Int. Math. Res. Notices 2010:13, 2532--2565.

\bibitem{eg2} Erdo\smash{\u{g}}an, M.~B.,  and Green, W.~R. \emph{Dispersive estimates for Schr\"{o}dinger operators in dimension two with obstructions at zero energy}.  Trans. Amer. Math. Soc. 365 (2013), 6403--6440.

\bibitem{eg3} Erdo\smash{\u{g}}an, M.~B., and Green, W.~R. \emph{A weighted dispersive estimate for Schr\"odinger operators in dimension two}. Commun. Math. Phys. 319 (2013), 791--811. 



\bibitem{egd} Erdo\smash{\u{g}}an, M.~B., and  Green, W.~R. \emph{The Dirac equation in two dimensions: Dispersive estimates and classification of threshold obstructions}.  Commun. Math. Phys. 352 (2017), no. 2, 719--757.

\bibitem{egd1} Erdo\smash{\u{g}}an, M.~B., and  Green, W.~R. \emph{On the one dimensional Dirac equation with potential}, J. Math. Pures Appl. (9) 151 (2021), 132--170. 

\bibitem{EGT} Erdo\smash{\u{g}}an, M.~B., Green, W.~R. and Toprak, E.  \emph{Dispersive estimates for Dirac operators in dimension three with obstructions at threshold energies.}  Amer. J. Math. 141, no. 5, 1217--1258.

\bibitem{EGT2} Erdo\smash{\u{g}}an, M.~B., Green, W.~R. and Toprak, E.  \emph{Dispersive estimates for massive Dirac operators in dimension two.}  J. Differential Equations 264 (2018), no. 9, 5802--5837.

\bibitem{EGTwhat} Erdo\smash{\u{g}}an, M.~B.,  Green, W.~R., and  Toprak, E.,
\emph{What is the Dirac equation?}   Notices Amer. Math. Soc. 68 (2021), no. 10, 1782--1785.	

\bibitem{ES} Erdo\smash{\u{g}}an, M.~B., and Schlag, W. \emph{Dispersive estimates for Schr\"odinger operators in the presence of a resonance and/or an eigenvalue at zero energy in dimension three: I}, Dynamics of PDE 1 (2004), 359--379.
 
\bibitem{EV}
Escobedo, M., and Vega, L. \emph{A semilinear Dirac equation in $H^s(\R^3)$ for $s>1$.}
SIAM J. Math. Anal. 28 (1997), no. 2, 338--362.

  
\bibitem{GM}
Georgescu, V., and Mantoiu, M. \emph{On the spectral theory of singular Dirac type Hamiltonians.}
J. Operator Theory 46 (2001), no. 2, 289-321. 

\bibitem{GV}  Ginibre, J., Velo, G.
\emph{Generalized Strichartz inequalities for the wave equation.}
J. Funct. Anal. 133 (1995), no.1, 50--68.


\bibitem{GS} Goldberg, M., and W.~Schlag. \emph{Dispersive estimates for Schr\"{o}dinger operators in dimensions one and three}. Comm. Math. Phys. vol. 251, no. 1 (2004), 157--178.

\bibitem{goldVis} Goldberg, M., and Visan, M. \emph{A Counterexample to Dispersive Estimates for Schrödinger Operators in Higher Dimensions.} Comm. Math. Phys. 266 (2006), no. 1, 211-238.
 


\bibitem{Jen}
Jensen, A. \emph{Spectral properties of Schr\"odinger operators and time-decay of the wave functions results in $L^2(R^m)$, $m\geq 5$}. Duke Math. J. 47 (1980), no. 1, 57--80.



\bibitem{JN}
Jensen, A., and G.~Nenciu. \emph{A unified approach to resolvent expansions at thresholds}.
Rev. Mat. Phys. vol. 13, no. 6 (2001), 717--754.



\bibitem{KS}
Krieger, J. and Schlag, W. \emph{Stable manifolds for all monic supercritical focusing nonlinear Schr\"odinger equations in one dimension} 
J. Amer. Math. Soc. 19 (2006), no. 4, 815--920.
 
\bibitem{KSW} Kraisler, J., Sagiv, A., and Weinstein, M. \emph{Dispersive decay estimates for Dirac equations with a domain wall.} Preprint 2023,  	arXiv:2307.06499.


\bibitem{Mur} Murata, M. {\em Asymptotic expansions in time for solutions of Schr\"odinger-type equations} J.\ Funct.\ Anal.~49 (1) (1982), 10--56.

\bibitem{PS}
Pelinovsky, D. and Stefanov, A. {\em Asymptotic stability of small gap solitons in nonlinear Dirac equations.}
J. Math. Phys. 53 (2012), no. 7, 073705, 27 pp. 

\bibitem{Sc2}
Schlag, W. \emph{Dispersive estimates for Schr\"{o}dinger operators in
dimension two.} Comm. Math. Phys. 257 (2005), no. 1, 87--117.

\bibitem{Scsurvey}  Schlag, W. \emph{On pointwise decay of waves.}  J. Math. Phys. 62 (2021), no. 6, Paper No. 061509, 27 pp.


\bibitem{Thaller}
Thaller, B. \emph{The Dirac equation.} Texts and Monographs in Physics. Springer-Verlag, Berlin, 1992.


\bibitem{Yam}
Yamada, O. \emph{A remark on the limiting absorption method for Dirac operators.}
Proc. Japan Acad. Ser. A Math. Sci. 69 (1993), no. 7, 243--246. 


\end{thebibliography}
\end{document}